\documentclass[12pt]{article}
\usepackage[utf8]{inputenc}
\usepackage[english]{babel}
\usepackage[letterpaper, margin=30mm, lmargin=30mm]{geometry}
\usepackage{amsmath}
\usepackage{amsthm}
\usepackage{amssymb}
\usepackage{graphicx}
\usepackage{tikz}
\usepackage{xy}
\usepackage[backend=bibtex,maxnames=99]{biblatex}
\usepackage{hyperref}
\usepackage{fancyhdr}
\usepackage{float}
\usepackage{extarrows}
\usepackage{setspace}
\usepackage[labelfont=bf]{caption}

\usetikzlibrary{positioning}
\tikzstyle{box} = [rectangle, rounded corners, text centered, draw=black, inner sep=7pt]

\hypersetup{
    colorlinks,
    citecolor=[rgb]{0 0 .6},
    filecolor=black,
    linkcolor=[rgb]{0 0 .6},
    urlcolor=[rgb]{0 0 .6}
}

\fancypagestyle{firstpage}{%

    \fancyfoot[L]{%
        \vspace*{-5\baselineskip}%
        \hrule%
        \vspace{0.5em}%
        \footnotesize This text was submitted in February, 2026 in fulfilment of the requirements for the %
        degree of Master of Science in Computer Science at the University of Auckland. %
        The project was supervised by Professor Andr\'e Nies.}
    \fancyfoot[C]{\vspace{-1\baselineskip}\thepage}
}

\theoremstyle{plain}
\newtheorem{theorem}{Theorem}[section]
\newtheorem{proposition}[theorem]{Proposition}
\newtheorem{lemma}[theorem]{Lemma}
\newtheorem{corollary}[theorem]{Corollary}
\newtheorem{question}[theorem]{Question}
\theoremstyle{definition}
\newtheorem{definition}[theorem]{Definition}
\theoremstyle{remark}

\newtheorem{remark}[theorem]{Remark}
\newtheorem{claim}[theorem]{Claim}

\newcommand{\hatc}{\widehat{\,\,\,}}
\newcommand{\wtow}{{}^\omega\omega}
\newcommand{\ttow}{{}^\omega2}
\newcommand{\infsets}{[\omega]^\omega}

\newcommand{\nonlow}{\mathrm{NonLow}}
\newcommand{\hypimm}{\mathrm{HypImm}}
\newcommand{\domfcn}{\mathrm{DomFcn}}
\newcommand{\opn}{\operatorname}
\newcommand{\as}{\mathcal{A}}
\newcommand{\ts}{\mathcal{T}}
\newcommand{\ps}{\mathcal{P}}
\newcommand{\af}{\mathcal{AF}}
\newcommand{\tf}{\mathcal{TF}}
\newcommand{\bd}{\mathcal{B}}
\newcommand{\cnv}{\!\downarrow\,}
\newcommand{\dvg}{\!\uparrow\,}
\newcommand{\ran}{\operatorname{ran}}
\newcommand{\massp}{\mathrm{MP}}
\newcommand{\inst}{\mathrm{inst}}
\newcommand{\sol}{\mathrm{sol}}
\newcommand{\card}{\opn{Card}}
\newcommand{\high}{\opn{H}}
\newcommand{\nl}{\opn{NL}}

\newcommand{\alt}[2]{#1_{(#2)}}
\newcommand{\alth}[2]{\hat#1_{(#2)}}
\newcommand{\bin}{\opn{bin}}
\newcommand{\add}{\opn{add}}
\newcommand{\cov}{\opn{cov}}
\newcommand{\non}{\opn{non}}
\newcommand{\cof}{\opn{cof}}
\newcommand{\notni}{{\not{\!\ni}}}

\numberwithin{equation}{section}

\title{\vspace*{1.5em}Maximal Eventually Different Families of Computable Functions}
\date{}
\author{Logan McDonald}

\begin{document}
    \maketitle

    \begin{abstract}
    Cardinal characteristics of the continuum are the cardinalities of interesting families of reals.
    A well-studied example is that of \emph{maximal almost disjoint (MAD)} families of sets of natural numbers.
    Significant work has been done investigating computability-theoretic analogues of cardinal characteristics.
    By considering encodings of MAD families as a single `universal' set, 
    Lempp, Miller, Nies, and Soskova (2023) studied the class of encoded MAD families.
    Such a class is referred to as a mass problem; one can study the relative complexity between mass problems.
    In Section 2, we build on the study of mass problems as analogues of cardinal characteristics.
    We define mass problems of \emph{maximal eventually different~(MED)} families of computable functions in the same way
    and compare them against the mass problems defined by Lempp et al. 
    In Section~3, we survey work by Greenberg, Kuyper, and Turetsky (2019) that provides an abstract framework for cardinal characteristics 
    and their effective counterparts.
    We show that this framework is suitable for obtaining results in the setting of mass problems.
    In Section 4, we showcase a construction by Schrittesser (2018) of an effectively closed MED family in set theory.
    We show that the construction is sufficiently effective that the computable members of the constructed family 
    are MED relative to computable functions.
\end{abstract}

    \thispagestyle{firstpage}

    \begin{figure}[b]
        \vspace*{17pt}
    \end{figure}

    \tableofcontents

    \section{Introduction}

The importance of the real numbers in much of modern mathematics needs no justification.
They were used without care until, in the 19th century, some groups began to work towards a rigorous foundation for all of mathematics.
As formal constructions of the real numbers developed, 
further questions could be asked about the analytic and topological properties of the real line.
At the turn of the 20th century, with Cantor's formal notion of the size of a set and proof of the uncountability of the reals,
the truth of the continuum hypothesis (CH) arose as the ultimate question about the reals at the most basic level:

\begin{quotation}
    \noindent There is no cardinality strictly between that of the natural numbers and that of the reals.
    That is, $\aleph_1=\mathfrak c$.
\end{quotation}

The definition of the constructible universe $L$ by G\"odel \cite{God:1939} and the method of forcing used by Cohen \cite{Coh:1963}
show that the Zermelo-Fraenkel axioms of set theory (ZFC), considered standard and sufficient for most of modern mathematics,
do not suffice to prove or disprove CH.
It is thus possible for models of ZFC to exist where there are uncountable subsets of the reals with sizes smaller than the continuum.
Since the failure of CH was proven consistent by Cohen \cite{Coh:1963}, 
a natural question has been to ask how far the universe can deviate from CH.
The technique used by Cohen can be easily adapted to show that there is no bound on the cardinality of the continuum in that for any ordinal $\alpha$,
there is a model of ZFC in which $\mathfrak c\ge\aleph_\alpha$.
Alternatively, we could look at the properties of cardinals below the continuum to compare with the case where CH holds.

Cardinal characteristics of the continuum, often referred to as simply \emph{cardinal characteristics} or occasionally \emph{cardinal invariants},
are uncountable cardinalities of subsets of the reals.
They usually correspond to the minimal sizes of subsets of some combinatorial, analytical, or topological interest.
Many developments in set-theoretic methods have come from showing the consistency of strict relations between cardinal characteristics,
constructing models of ZFC which are in some sense far from satisfying~CH.

As computability theory has developed alongside set theory, many parallels have been drawn between the two subjects.
The first approach to computability-theoretic analogues of cardinal characteristics was in the form of highness classes of oracles.
Highness classes represent the minimal required complexity in order to hold some property, 
analagous to cardinal characteristics defined as minimal cardinalities of families with some property.

\sloppy Rupprecht \cite{Rup:2010-1}, under the supervision of Blass, 
defined analogues for a well-known collection of cardinal characteristics: Cicho\'n's diagram.
Brendle, \mbox{Brooke-Taylor}, Nies, and Ng \cite{BBNN:2015} built on this framework and answered various questions.
All of the set-theoretic relations were reflected in their computability-theoretic counterparts,
but there were also several additional relations that only hold in computability, causing several of the highness class analogues to coincide.

This work proved useful in other computability research.
Monin and Nies~\cite{MN:2015} took the Gamma question, 
a problem in coarse computability concerning the spectrum of possible Gamma values of Turing degrees,
and reduced it to a problem concerning analogues of cardinal characteristics.
Following this, Monin \cite{Mon:2018} obtained a solution.
Later work by Monin and Nies \cite{MN:2021} began considering highness classes as their Muchnik degrees and obtained more general results.
We expand on this in Section \ref{analogue}.

Using a well-studied framework for abstractly representing various cardinal characteristics,
Greenberg, Kuyper, and Turetsky \cite{GKT:2019} show that this framework is also suitable for 
defining and showing results about highness classes as analogues of cardinal characteristics.
They show that most of the results by Rupprecht \cite{Rup:2010-1} and from \cite{BBNN:2015} can be systematically obtained in this setting.

The framework of highness classes as an analogue of cardinal characteristics is only suitable for those cardinal characteristics
which can be represented as the cardinalities of unbounded or cofinal subsets in relational structures.
Considering the possibility of analogues of other cardinal characteristics, 
Lempp, Miller, Nies, and Soskova \cite{LMNS:2023} developed a framework of 
analogues of the underlying families used to define cardinal characteristics.
This setting is much more flexible than that of highness classes, 
allowing for the possibility of studying most combinatorial cardinal characteristics.

This thesis will largely focus on computability-theoretic analogues of families of functions that are of interest to 
cardinal characteristics research in set theory.
In Section \ref{sf}, we define analogues of maximal eventually different families and maximal towers of functions 
in the framework developed in \cite{LMNS:2023}.
and relate them to previously studied analogues for subsets of Cantor space.
Section \ref{mpcichon} surveys work in \cite{GKT:2019} 
and applies it to a more recent framework for analogues of cardinal characteristics.
In Section \ref{schrittesser}, we expand on a construction by Schrittesser \cite{Sch:2018} of an effectively closed maximal eventually different family
in set theory, and show that the construction also yields an analogous result in computability theory.

\subsection{Cardinal characteristics of the continuum}

CH is stronger than necessary to control much of the structure below the continuum.
If every cardinal below the continuum behaves in some sense like $\aleph_0$, then the universe is in some sense not far from CH.
This is the idea encapsulated in Martin's Axiom:

\begin{definition}[MA$(\kappa)$]
    For every partial order $\mathbb P$ which satisfies the countable chain condition (c.c.c.)\
    and any family $D$ of dense subsets of $\mathbb P$ with $|D|\le\kappa$,
    there is a filter $F\subseteq\mathbb P$ such that $F$ meets every member of $D$.
\end{definition}

\begin{definition}[MA]
    Let $\mathfrak m$ be the least cardinal $\kappa$ such that MA$(\kappa)$ is false.
    MA is the statement $\mathfrak m=\mathfrak c$.
\end{definition}

MA$(\aleph_0)$ is provable and MA$(\mathfrak c)$ can be refuted in ZFC, so $\aleph_1\le\mathfrak m\le\mathfrak c$ and CH implies MA.
Even in models where the continuum is large, MA can still hold,
and so many useful properties of countable sets would be generalised to all uncountable sets of reals smaller than the continuum.
For example, if we assume MA, then every subset of the reals with cardinality smaller than $\mathfrak c$ has measure 0 and is meagre.
In general, for a cardinal characteristic $\mathfrak x$, 
the assumption $\mathfrak x=\mathfrak c$ is some way of imposing a structural constraint on the reals.

We can find simple and natural examples of cardinal characteristics in Baire space:

\begin{definition}[$\mathfrak b$]
    A family $F\subseteq\wtow$ is called \emph{$\le^*$-unbounded} if there is no function $h$ such that $f\le^*h$ for all $f\in F$.
    Define
        $$\mathfrak b=\min\{|F|:F\text{ is}\le^*\!\!\text{-unbounded}\}.$$
\end{definition}

\begin{definition}[$\mathfrak d$]
    A family $F\subseteq\wtow$ is called \emph{$\le^*$-dominating} if for every function $h$ there is $f\in F$ such that $h\le^*f$.
    Define
        $$\mathfrak d=\min\{|F|:F\text{ is}\le^*\!\!\text{-dominating}\}.$$
\end{definition}

Every $\le^*$-dominating family is $\le^*$-unbounded and $\wtow$ is $\le^*$-dominating, so $\mathfrak b\le\mathfrak d\le\mathfrak c$.
As with many cardinal characteristics, a simple diagonalisation argument shows that $\mathfrak b$ and $\mathfrak d$ are uncountable:
Suppose there was a countable unbounded family $\{f_i\}_{i\in\omega}$, the function $x\mapsto\max_{i\le n}f_i(n)$ dominates each $f_i$.
Both $\mathfrak b$ and $\mathfrak d$ were considered before the independence proof of CH, 
by Rothberger \cite{Rot:1938} and Kat\v etov \cite{Kat:1960} respectively.

Set theorists became increasingly interested in the topological and analytic properties of the real line.
Over time, it became apparent that hypotheses regarding the analytical properties of the real line tended to be stronger than topological ones.
Shelah \cite{She:1984} showed that ZFC proves the consistency of every projective set having the property of Baire,
but a large cardinal assumption is required for the consistency of every projective set being measurable.

Bartoszy\'nski \cite{Bar:1987} looked at the following concepts.
For the following, $\mathcal I$ is an ideal on an infinite set $X$ which contains all finite subsets of $X$.

\begin{definition}[$\add(\mathcal I)$]
    The additivity number of an ideal $\mathcal I$ is the least cardinality of a family of members of $\mathcal I$
    whose union is not a member of $\mathcal I$:
        $$\add(\mathcal I)=\min\{|F|:F\subseteq\mathcal I\text{ and }\bigcup F\notin\mathcal I\}.$$
\end{definition}

\begin{definition}[$\cov(\mathcal I)$]
    The covering number of an ideal $\mathcal I$ is the least cardinality of a family of members of $\mathcal I$
    whose union is $X$:
        $$\cov(\mathcal I)=\min\{|F|:F\subseteq\mathcal I\text{ and }\bigcup F=X\}.$$
\end{definition}

\begin{definition}[$\non(\mathcal I)$]
    The uniformity number of an ideal $\mathcal I$ is the least cardinality of a subset of $X$ which is not in $\mathcal I$:
        $$\non(\mathcal I)=\min\{|Y|:Y\subseteq X\text{ and }Y\notin\mathcal I\}.$$
\end{definition}

\begin{definition}[$\cof(\mathcal I)$]
    The cofinality of an ideal $\mathcal I$ is the least cardinality of a family of members of $\mathcal I$
    which is cofinal in the ordering $\langle\mathcal I,\subseteq\rangle$:
        $$\cof(\mathcal I)=\min\{|F|:F\subseteq\mathcal I\text{ and }(\forall Y\in\mathcal I)(\exists Z\in F)[Y\subseteq Z]\}.$$
\end{definition}

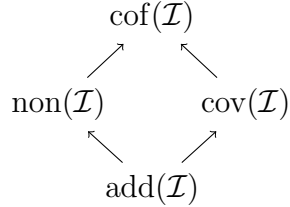
\begin{figure}[h]
    \centering
    \begin{tikzpicture}[node distance = 1em and 1.5em]
        \node (addI) {$\add(\mathcal I)$};
        \node (covI) [above = of addI, xshift = 3em] {$\cov(\mathcal I)$};
        \node (nonI) [above = of addI, xshift = -3em] {$\non(\mathcal I)$};
        \node (cofI) [above = of addI, yshift = 2.8em] {$\cof(\mathcal I)$};

        \draw[->] (addI) -- (covI);
        \draw[->] (addI) -- (nonI);
        \draw[->] (covI) -- (cofI);
        \draw[->] (nonI) -- (cofI);
    \end{tikzpicture}
    \caption{Relations between cardinal characteristics associated with an ideal.}
\end{figure}

Of interest to the study of cardinal characteristics is the case where $|X|=\mathfrak c$ and $\mathcal I$ is a $\sigma$-ideal, that is, 
it is closed under countable unions.
This ensures that $\aleph_1\le\add(\mathcal I)$ and usually that $\cof(\mathcal I)\le\mathfrak c$.
Two such ideals of general importance in mathematics are the ideals of meagre and Lebesgue null sets of reals.
We refer to them as $\mathcal M$ and $\mathcal N$ respectively.
Relating the cardinal characteristics associated with these ideals to each other and to $\mathfrak b$ and $\mathfrak d$,
we obtain Figure \ref{cichondiagram}.
It was named after Cicho\'n by Fremlin \cite{Fre:1984}.

\begin{figure}[h]
    \centering
    \begin{tikzpicture}[node distance = 1em and 1.5em]
        \node (addM) {$\mathrm{add}(\mathcal M)$};
        \node (covM) [right = of addM] {$\mathrm{cov}(\mathcal M)$};
        \node (b) [above = of addM] {$\mathfrak b$};
        \node (d) [above = of covM] {$\mathfrak d$};
        \node (nonM) [above = of b] {$\mathrm{non}(\mathcal M)$};
        \node (cofM) [above = of d] {$\mathrm{cof}(\mathcal M)$};
        \node (addN) [left = of addM] {$\mathrm{add}(\mathcal N)$};
        \node (covN) [left = of nonM] {$\mathrm{cov}(\mathcal N)\,$};
        \node (nonN) [right = of covM] {$\mathrm{non}(\mathcal N)$};
        \node (cofN) [right = of cofM] {$\,\mathrm{cof}(\mathcal N)\;$};
        \node (a1) [left = of addN] {$\aleph_1$};
        \node (c) [right = of cofN] {$\mathfrak c$};

        \draw[->] (addN) -- (addM); \draw[->] (addN) -- (covN); \draw[->] (addM) -- (covM); \draw[->] (addM) -- (b);
        \draw[->] (b) -- (nonM); \draw[->] (covM) -- (d); \draw[->] (covN) -- (nonM); \draw[->] (d) -- (cofM);
        \draw[->] (b) -- (d); \draw[->] (nonM) -- (cofM); \draw[->] (covM) -- (nonN);
        \draw[->] (nonN) -- (cofN); \draw[->] (cofM) -- (cofN); \draw[->] (a1) -- (addN); \draw[->] (cofN) -- (c);
    \end{tikzpicture}
    \caption{Cicho\'n's diagram.}
    \label{cichondiagram}
\end{figure}
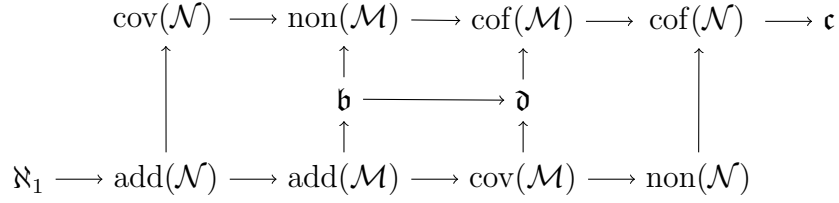

Celebrated work by Goldstern et al.\ \cite{GKMS:2022} has shown that it is consistent that the maximum separation occurs in Figure \ref{cichondiagram},
that all of the cardinals in \mbox{Cicho\'n's} diagram have distinct values, with the exception of the dependencies 
${\add(\mathcal M)=\min\{\cov(\mathcal M),\mathfrak b\}}$ and ${\cof(\mathcal M)=\max\{\non(\mathcal M),\mathfrak d\}}$.

A key feature of the cardinals in Cicho\'n's diagram, and some other cardinal characteristics, 
is that there is no internal structure or restrictions on which reals the witnessing families can contain.
This allows for a very general framework of these cardinal characteristics, which will be expanded on in Section \ref{gktintro}.
Investigating families with internal structure gives more complex cardinal characteristics.
The following are two of the most well-studied such cardinal characteristics:

\begin{definition}[$\mathfrak a$]
    A family $G\subseteq\infsets$ is \emph{almost disjoint (AD)} if no two of its members have infinite intersection.
    The family $G$ is \emph{maximal almost disjoint (MAD)} if there is no AD family which strictly contains it.
    Define
        $$\mathfrak a=\min\{|G|:G\text{ is MAD}\}.$$
\end{definition}

\begin{definition}[$\mathfrak t$]
    A family $G\subseteq\infsets$ is called a \emph{tower} if it is linearly ordered by $\subseteq^*$.
    The tower $G$ is \emph{maximal} if there is no set $R$ such that $R\subseteq^*X$ for all $X\in G$.
    Define
        $$\mathfrak t=\min\{|G|:G\text{ is a maximal tower}\}.$$
\end{definition}

\subsection{Nonlowness classes and mass problems}
\label{analogue}

The rich combinatorial nature of many cardinal characteristics makes them an appealing target for possible computability-theoretic analogues.
In particular, many cardinal characteristics are defined in terms of an object with a universal property:
for every real $x$, some relation holds. This kind of relation is easy to adapt to computability by restricting to computable reals $x$.

Computability theoretic analogues of cardinal characteristics were first considered in the doctoral research of Rupprecht \cite{Rup:2010-1}.
He defined `Turing characteristics' as highness classes of oracles, in particular ones that correspond to the cardinals in Cicho\'n's diagram.
In \cite{BBNN:2015}, this theory was developed further and made Rupprecht's results more accessible.
They give a full diagram of analogues to Cicho\'n's diagram.

The analogue of a cardinal characteristic in this framework is related to the method of forcing 
that would be used to increase the value of the characteristic.
The cardinal $\mathfrak b$ is the least cardinality of a family of functions which is not uniformly dominated by a single function.
In order to create a model with a greater value of $\mathfrak b$, we destroy this property of these families, 
so we add a function that dominates each one in the ground model.
In the set theory setting, we need to iterate this process to actually increase the cardinality of $\mathfrak b$.
The computability-theoretic analogue is a function that dominates every computable function, a dominating function.
Relativisation to a dominating function is analogous to working in a forcing extension 
that contains a function that dominates those in the ground model.

Analogues for the rest of Cicho\'n's diagram are obtained in a similar way.
The analogue for $\mathfrak d$ is a function which is not dominated by any computable function: a hyperimmune function.
The cardinals associated with the ideals of meagre and null sets are more complicated,
but by identifying meagre or null sets with reals that code them, we can obtain similar highness classes.
For example, $\mathrm{add}(\mathcal M)$ is the least cardinality of a family of meagre sets whose union is not meagre.
To destroy that property of a family, we add a meagre set which contains all ground model meagre sets.
Thus, the analogue is a code for a meagre set that contains every meagre set which coded by computable reals.
An oracle which computes such a code is called \emph{meagre engulfing}.

Unfortunately, some of the analogues defined by Rupprecht \cite{Rup:2010-1} or in \cite{BBNN:2015} turn out to be equivalent.
One such case is that the analogues for $\mathfrak b$ and for $\opn{add}(\mathcal M)$;
dominating functions and meagre engulfing oracles both characterise the high degrees.
It is then natural to ask if we can find a framework that more closely reflects the relations in set theory.
Kihara \cite{Kih:2017} considered a setting like this for the hyperarithmetical sets as opposed to the computable ones.
Kihara shows that the hyperarithmetical-theoretic analogues for $\opn{cov}(\mathcal M)$ and $\mathfrak d$ can be separated
while the computability-theoretic ones can not.

In \cite{BBNN:2015}, the application of this framework to analogues of cardinals outside of Cicho\'n's diagram is also summarised.
They give analogues of the cardinal characteristics $\mathfrak s$ and $\mathfrak r$ as $r$-cohesive and bi-immune degrees and relate them to the
analogues of Cicho\'n's diagram.
Valverde and Tveite \cite{VT:2021} considered analogues of the cardinal characteristics $\mathfrak e$ and $\mathfrak v$ by defining
prediction and evasion degrees.
They also show how these fit with the previously studied analogues.

Monin and Nies \cite{MN:2015} investigated a problem in the study of coarse computability: the \emph{Gamma question}.
The Gamma question concerns the spectrum of possible \emph{Gamma values} of Turing degrees.
The reader can find details in the aforementioned paper by Monin and Nies.
They look at the highness classes studied by Rupprecht \cite{Rup:2010-1} and in \cite{BBNN:2015}
and the Gamma values of their members.
Further work on this by Monin \cite{Mon:2018} used this approach and resolved the Gamma question, showing the only possible Gamma values, 
and those which are attained, are $0$, $1/2$, and $1$.

A different analogue of cardinal characteristics is described in later work by Monin and Nies \cite{MN:2021},
providing a more general framework for problems like the Gamma question.
They define mass problems, classes of sets or functions, and we can study the relative complexity in terms of the mass problems themselves.
Each highness class previously studied can be considered as a mass problem.
The Gamma question can be reduced to showing Muchnik or Medvedev equivalence between some mass problems.

Given mass problems $\mathcal X$ and $\mathcal Y$, we say that $\mathcal X$ is \emph{Muchnik (weakly) reducible} to $\mathcal Y$, 
written $\mathcal X\le_w\mathcal Y$, if each element of $\mathcal Y$ computes a member of $\mathcal X$.
If this computation is uniform, that is, there is a Turing functional $\Gamma$ such that $\Gamma^y\in\mathcal X$ for any $y\in\mathcal Y$, 
then we say that $\mathcal X$ is \emph{Medvedev (strongly) reducible} to $\mathcal Y$, written $\mathcal X\le_s\mathcal Y$.
We will often talk about the Muchnik framework as it is the most general.
Most reductions we prove will be sufficiently uniform for the Medvedev relation to hold also.

The setting of mass problems in the Muchnik ordering is very appealing for analogues of cardinal characteristics.
Cardinal characteristics are often defined as the minimum of the spectrum of cardinalities of a kind of family.
The Muchnik degrees act as an extension of the Turing degrees to be closed under infima and suprema:
The Turing degrees embed into the Muchnik degrees by taking the singleton of any of its members as a representative,
and the union of two Muchnik degrees is their infimum.

In \cite{LMNS:2023}, mass problems are also used for analogues of cardinal characteristics.
Rather than looking at the forcing used to change the value of the cardinals, they looked at the families that witness the cardinality.
They consider a mass problem, a class of sets or functions, containing all families with a property.
Working in computability, we need only consider countable families.
As such, we can code the family as a single object:

\begin{definition}
    \label{seqcode}
    Given a function $f:\omega\to\omega$, we define the columns of $f$ by $f^{[e]}:x\mapsto f(\langle x,e\rangle)$.
    Similarly, for a set $F\subseteq\omega$, the columns of $F$ are defined as $F^{[e]}=\{x:\langle x,e\rangle\in F\}$.
    This allows us to code a sequence $\langle x_e\rangle_{e\in\omega}$ of functions or sets
    as a single universal object $x$ where $x^{[e]}=x_e$ for each $e$.
\end{definition}

We can use this to define analogues of cardinal characteristics as a mass problem of codes for sequences with some property.

In \cite{LMNS:2023}, analogues of the cardinal characteristics 
$\mathfrak a$, $\mathfrak t$, $\mathfrak u$, and $\mathfrak i$ are defined in this way.
We will focus on their analogues of $\mathfrak a$ and $\mathfrak t$:

\begin{definition}[$\as$]
    \label{asdef}
    A sequence $\langle G_e\rangle_{e\in\omega}$ of computable sets is called \emph{almost disjoint} 
    if its members have pairwise finite intersections.
    Such a family is \emph{maximal almost disjoint (MAD)} if for any computable set $H$ there is some $e$ such that $G_e\cap H$ is infinite.
    Let $\as$ be the mass problem of MAD families.
\end{definition}

\begin{definition}[$\ts$]
    \label{tsdef}
    A sequence $\langle G_e\rangle_{e\in\omega}$ of infinite computable sets is called a \emph{tower} if we have for every $e$ that
    $G_{e+1}\subseteq^*G_e$ and $G_e\smallsetminus G_{e+1}$ is infinite.
    Such a tower is \emph{maximal} if for any infinite computable set $H$ there is some $e$ such that $H\smallsetminus G_e$ is infinite.
    Let $\ts$ be the mass problem of maximal towers.
\end{definition}

It was shown in \cite[Fact 2.2]{LMNS:2023} that these mass problems are Medvedev equivalent.
In Section \ref{sf}, we will give a weaker alternative to $\ts$ that could better reflect the set-theoretic relations between cardinal characteristics.
The author's previous work \cite{McD:2024} built on the framework in \cite{LMNS:2023}
and defined an analogue for maximal ideal independent families of sets.
They also consider the complexity of maximal independent and maximal ideal independent families in the cases of other Boolean algebras of sets,
the $\omega$-computably approximable and $K$-trivial sets respectively.

Mass problems with the Muchnik reduction can be argued to be a more natural structure than the Turing degrees to study computational complexity.
By choosing a canonical representative for each Muchnik degree, the union of its members, we can see that the Muchnik degrees form a lattice:
For any two Muchnik degrees, their join can be given as their union, and their meet can be given as their intersection.

Both nonlowness classes and mass problems provide correspondents of cardinal characteristics
which somewhat faithfully reflect the relations in set theory.
More equivalences arise in these computability-theoretic settings, which suggests that 
there is often an inherent complexity in separating cardinal characteristics.
Kihara's work \cite{Kih:2017} on analogues in hyperarithmetical theory reinforces this.
The mass problem analogy is more general and more closely reflects the families in set theory.

\subsection{Weihrauch problems and morphisms}
\label{gktintro}

Weihrauch problems provide an abstract structure that allows for the representation of various cardinal characteristics.
As this structure is not tied to the set-theoretic nature of the cardinal characteristics themselves,
it is also useful in defining computability-theoretic analogues.
We give an exposition on this framework here; it will be referenced in Section \ref{mpcichon}.

Fremlin \cite{Fre:1993} observed that various cardinal characteristics, including those in Cicho\'n's diagram,
can be expressed as the least size of a set which witnesses the totality of some relation.
For example, the relation associated with $\mathfrak d$ is $\le^*$ on elements of $\wtow$.
We have that $\mathfrak d$ is the least size of a set of functions $F$ such that for any $h\in\wtow$ there is $f\in F$ such that $f\le^*h$.
By substituting a different relation in place of $\le^*$, we obtain various other cardinal characteristics.

Vojt\'a\v{s} \cite{Voj:1993} came up with a general framework of this form.
He studied reductions between these relational structures and gave them the name \emph{generalised Galois-Tukey connections}.
Part of the motivation for this framework is that these relational structures and their Galois-Tukey connections 
can be studied independently of the cardinalities associated with them, and so they are still interesting to study even if CH holds.

We follow the treatment of \cite{GKT:2019}.
They survey the usage of these relations in the study of cardinal characteristics in both set theory and computability,
using terminology which is widespread due to its independent development in computable analysis and reverse mathematics:

\begin{definition}
    A \emph{Weihrauch problem} is defined as a triple $A=\langle A_\inst,A_\sol,A\rangle$ 
    where $A_\inst,A_\sol\subseteq\wtow$ and $A\subseteq A_\inst\times A_\sol$.
    Elements of $A_\inst$ are referred to as \emph{instances} and elements of $A_\sol$ are referred to as \emph{solutions}. 
    We say that $b\in A_\sol$ is a solution for $A\in A_\inst$ if $aAb$.
\end{definition}

The simplest examples are the Weihrauch problems we associate with $\mathfrak d$ and~$\mathfrak b$:
Let $\opn{Dom}=\langle \wtow,\wtow,\le^*\rangle$ and $\opn{Esc}=\langle\wtow,\wtow,<^\infty\rangle$.
The cardinal characteristic associated with such a problem is defined in a way that generalises the correspondence we have between these
problems and cardinal characteristics:

\begin{definition}
    A \emph{complete solution set} for a Weihrauch problem $A$ is one that for each instance $a\in A_\inst$, contains a solution $b\in A_\sol$ for $a$.
    The cardinal characteristic associated with $A$ is defined as:
        $$\card(A)=\min\{|C|:C\text{ is a complete solution set for }A\}.$$
\end{definition}

We see immediately from the definitions that $\mathfrak d=\card(\opn{Dom})$ and $\mathfrak b=\card(\opn{Esc})$.
In \cite{GKT:2019}, it is observed that the flexibility in this definition also allows for abstraction of the
computability-theoretic analogues for cardinal characteristics defined by Rupprecht \cite{Rup:2010-1} and in \cite{BBNN:2015}:

\begin{definition}
    Given a Weihrauch problem $A$, define $$A_\inst^r=\{a\in A_\inst:a\text{ computable}\},$$ and $A_\sol^r$ similarly.
    The highness class associated with $A$ is defined as:
        $$\high(A)=\{Z:(\exists b\le_TZ)(\forall a\in A_\inst^r)[aAb]\}.$$
\end{definition}

We have that $\high(\opn{Dom})$ is the class of oracles which compute dominating functions, 
and $\high(\opn{Esc})$ is the class of oracles which compute hyperimmune functions.
The reader may notice that the order is changed from what was proposed by Rupprecht \cite{Rup:2010-1} and in \cite{BBNN:2015}.
The forcing analogy is dual to the associated highness class; 
in \cite{GKT:2019}, it is suggested that the nonlowness class via the dual problem may be more natural to consider:

\begin{definition}
    Given a Weihrauch problem $A=\langle A_\inst,A_\sol,A\rangle$, define its \emph{dual} $A^\bot=\langle A_\sol,A_\inst,A^\bot\rangle$ where
        $$bA^\bot a\iff\lnot aAb.$$
    The associated nonlowness class associated with $A$ is $\nl(A)=\high(A^\bot)$, explicitly:
        $$\nl(A)=\{Z:(\exists a\le_TZ)(\forall b\in A_\sol^r)[\lnot aAb]\}.$$
\end{definition}

Then we have $\nl(\opn{Dom})$ as the class of oracles which compute hyperimmune functions, 
and $\nl(\opn{Esc})$ those which compute dominating functions.
It should be noted that cardinal characteristics and highness classes themselves do not inherently have duals.
There are no canonical Weihrauch problems associated with these objects;
we use the ones most convenient concerning relations between problems and duality.
As an example, $\mathfrak d=\card(\wtow,\wtow,\le)$ but the dual of this Weihrauch problem has associated cardinal $\aleph_0$.

Blass \cite{Bla:2010} used the term \emph{morphism} in place of the term Galois-Tukey connection used by Vojt\'a\v{s} \cite{Voj:1993}.
The following effective notion is proposed in \cite{GKT:2019}, ensuring it also gives relations between highness or nonlowness classes:

\begin{definition}
    An \emph{effective morphism} $\varphi$ from Weihrauch problems $A$ to $B$ is a pair of hyperarithmetic piecewise computable maps 
    $\varphi_\inst:A_\inst\to B_\inst$ and $\varphi_\sol:B_\sol\to A_\sol$ such that whenever $a\in A_\inst$, $b\in B_\sol$, and $\varphi_\inst(a)Bb$,
    then $aA\varphi_\sol(b)$.
    Write $A\to B$ if there is an effective morphism from $A$ to $B$.
\end{definition}

\begin{figure}[h]
    \begin{tikzpicture}
        \node(b) {$B_{\mathrm{inst}},B_{\mathrm{sol}}$};
        \node(a) [below of = b] {$A_{\mathrm{inst}},A_{\mathrm{sol}}$};

        \draw (a.west) edge[->] [out=180,in=180] (b.west) node[xshift=-0.8cm,yshift=0.5cm] {$\varphi_{\mathrm{inst}}$};
        \draw (b.east) edge[->] [out=0,in=0] (a.east) node[xshift=0.8cm,yshift=-0.5cm] {$\varphi_{\mathrm{sol}}$};
    \end{tikzpicture}
    \centering
    \label{effmorph}
    \caption{An effective morphism}
\end{figure}
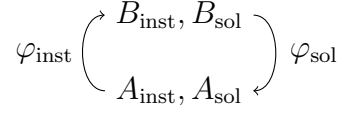

An effective morphism implies the relations we want in both the set-theoretic and computability-theoretic settings:

\begin{proposition}
    Let $A$ and $B$ be Weihrauch problems. Suppose $A\to B$. Then
    \begin{itemize}
        \item $\card(A)\le\card(B)$, 
        \item $\high(B)\subseteq\high(A)$, and
        \item $\nl(A)\subseteq\nl(B)$.
    \end{itemize}
\end{proposition}

This is very much not an equivalence.
Some relations between associated cardinal characteristics will hold where there is no effective morphism between the Weihrauch problems.
It is even possible that two Weihrauch problems share an associated cardinal characteristic without an effective morphism, 
see \mbox{\cite[Lemma 4.22]{GKT:2019}}.
As demonstrated by the collapsing of relations in Cicho\'n's diagram demonstrated in \cite{BBNN:2015},
relations in the computability setting also do not imply the existence of an effective morphism.
There are often techniques in computability that make heavy use of the countability of the setting, or that we can uniformly approximate
every computable function.

Conventions for and clarifications of notation can be found in Appendix \ref{appnot}.

    \newpage
    \section{Mass Problems of Sequences of Functions}
\label{sf}

The almost disjointness number $\mathfrak a$ and the tower number $\mathfrak t$, Definitions \ref{asdef} and \ref{tsdef} respectively,
are two of the most well-researched cardinal characteristics.
Replacing sets in the definitions of $\mathfrak a$ and $\mathfrak t$ with functions, we can obtain two more cardinal characteristics:

\begin{definition}
    \label{med}
    A family $G\subseteq\wtow$ is called \emph{eventually different} (ED) if any pair of elements agree only on a finite set of inputs. 
    The family $G$ is maximal (MED) if it is maximal among such families under inclusion.
    Define:
        $$\mathfrak a_e=\min\{|G|:G\subseteq\wtow\text{ is MED}\}.$$
\end{definition}
\begin{definition}
    A family $G\subseteq\wtow$ is called a \emph{tower} if it is linearly ordered by $\ge^*$. 
    The family $G$ is maximal if there is no function $f$ such that $f\le^*g$ for every $g\in G$.
    Define:
        $$\mathfrak t_e=\min\{|G|:G\subseteq\wtow\text{ is a maximal tower}\}.$$
\end{definition}

The cardinal $\mathfrak a_e$ was introduced by Zhang \cite{Zha:1999} based on a suggestion by B. Veli\v{c}kovi\'c.
Blass et al.\ \cite{BHZ:2007} give a survey of some fundamental results about $\mathfrak a_e$ and related cardinal characteristics.
Towers in $\wtow$ were first studied by Dordal \cite{Dor:1989} and related to several cardinal characteristics in $\infsets$.

Lempp, Miller, Nies, and Soskova \cite{LMNS:2023} introduce mass problems as an analogy for cardinal characteristics 
in the setting of computability.
They give some examples in mass problems of sequences of computable sets.
We develop this further and study mass problems of sequences of computable functions, specifically, analogues of the cardinal characteristics 
$\mathfrak b,\mathfrak a_e$, and $\mathfrak t_e$.
The analogues as mass problems will be called $\bd,\af$, and $\tf$, respectively.
Their members will be sequences of functions coded as a single universal function, as in Definition \ref{seqcode}.
We relate these to the mass problems $\ts$ and $\as$ of sequences of sets studied in \cite{LMNS:2023}.

\begin{figure}[h]
    \centering
    \begin{tikzpicture}[node distance=3.5em]
        \node[box] (b) {$\smash{\bd\xlongequal{\text{\ref{af=b}}}}\af$};
        \node[box, below of=b] (t) {$\smash{\ts\xlongequal{\text{\cite[$2.2$]{LMNS:2023}}}}\as$};
        \node[box, below of=t, xshift=-3.5em] (p) {$\smash{\ps\xlongequal{\text{\ref{t-=p}}}}\ts^-$};
        \node[box, below of=t, xshift=2.5em] (tf) {$\tf$};
        \node[box, left of=b, xshift=-2.5em] (nl) {$\nonlow$};
        \node[box, right of=t, xshift=2.8em] (hi) {$\hypimm$};

        \draw (b) -- node[anchor=west,xshift=-2pt]{\scriptsize\ref{t<b}} (t) ; 
        \draw[densely dotted] (t) -- (p); 
        \draw[densely dotted] (t) -- node[anchor=west,xshift=1pt]{\scriptsize\ref{tf<t}} (tf);
        \draw (nl) -- node[anchor=east,xshift=-4pt]{\scriptsize\cite[$3.1$]{LMNS:2023}} (t);
        \draw (b) -- node[anchor=west,xshift=3pt]{\scriptsize\ref{hypimm<bd}} (hi);
    \end{tikzpicture}
    \caption{Muchnik relations between mass problems. Dotted lines indicate relations not known to be strict.}
\end{figure}
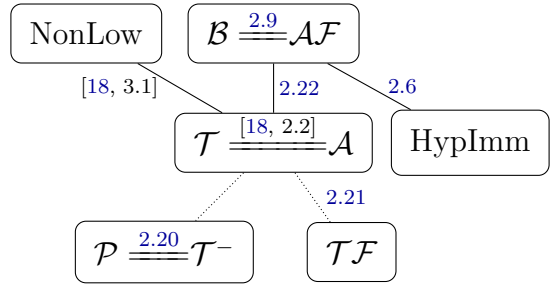

\setlength{\tabcolsep}{0.8em} {\renewcommand{\arraystretch}{1.4}
\begin{figure}[h]
    \centering
    \begin{tabular}{|l|l|l|}
        \hline
        $\af$       & Def. \ref{afdef}  & MED families of functions. \\ \hline
        $\tf$       & Def. \ref{tfdef}  & Maximal towers of functions. \\ \hline
        $\bd$       & Def. \ref{bddef}  & $\le^*$-unbounded families of functions. \\ \hline
        $\as$       & Def. \ref{asdef}  & MAD families of sets. \\ \hline
        $\ts$       & Def. \ref{tsdef}  & Maximal towers of sets. \\ \hline
        $\ts^-$     & Def. \ref{t-def}  & Maximal weak towers of sets. \\ \hline
        $\ps$       & Def. \ref{psdef}  & Families with the pseudointersection property. \\ \hline
    \end{tabular}
    \caption{Summary of mass problems.}
\end{figure}

Through observing relations to well-studied notions such as hyperimmune and non-low sets, we can obtain separations between some of these mass problems.
This contrasts work in \cite{LMNS:2023} and by the author \cite{McD:2024} on mass problems of some necessarily more complex objects, 
which all coincide with high sets in the Muchnik degrees.

In addition to studying the mass problems, we get results on the descriptive complexity of the sequences of functions themselves.

\subsection{Analogues of $\mathfrak a_e$ and $\mathfrak t_e$}
\label{analoguesaftf}

Recall that in \cite{LMNS:2023}, analogues are defined for the cardinal characteristics $\mathfrak a$ and $\mathfrak t$ as mass problems
consisting of sequences with specific properties.
In the same way, we define the following analogues for $\mathfrak a_e$ and $\mathfrak t_e$:

\begin{definition}[$\af$]
    \label{afdef}
    A sequence $\langle g_e\rangle_{e\in\omega}$ of computable functions is called \emph{eventually different} 
    if its members are pairwise eventually different.
    Such a family is \emph{maximal (MED)} if for any computable function $h$ there is some $e$ such that $g_e=^\infty h$.
    Let $\af$ be the mass problem of MED families.
\end{definition}
\begin{definition}[$\tf$]
    \label{tfdef}
    A sequence $\langle g_e\rangle_{e\in\omega}$ of unbounded computable functions is called a \emph{tower} if
    $g_{e+1}\le^* g_e$ and $g_{e+1}<^\infty g_e$ for every $e$.
    Such a tower is \emph{maximal} if for any unbounded computable function $h$ such that $g_e<^\infty h$ for some $e$.
    Let $\tf$ be the mass problem of maximal towers.
\end{definition}

The notion of an MED family is directly analogous to MAD families of sets.
Given a family of functions, we can consider the corresponding family of sets obtained by taking the graph of each function.
The family of graphs is AD if and only if the family of functions is ED.
The family of graphs is maximal among AD families of graphs if and only the family of functions is MED.
Note that there is no direct correspondence; not every set is identified with the graph of a function.
A similar analogy in the other direction can be obtained by taking the characteristic functions of members of a family of sets.

The notion of a tower of functions is more loosely associated with the notion of a tower of sets.
Each value can potentially decrease many times without ever increasing, rather than only once.
The relations between these classes of families of functions and those of families of sets will be investigated further in \mbox{Section \ref{seqsets}}.

We do not need to only consider mass problems of families with internal structure.
We define an analogue for $\mathfrak b$ in the same way:

\begin{definition}[$\bd$]
    \label{bddef}
    A sequence $\langle f_e\rangle_{e\in\omega}$ of computable functions is called \emph{$\le^*$-unbounded} if there is no computable function $h$
    which dominates every $f_e$.
    Let $\mathcal B$ be the mass problem of $\le^*$-unbounded sequences.
\end{definition}

A clear example is that any sequence which contains every computable function is $\le^*$-unbounded:
for each $h$ computable, $h+1$ is not dominated by $h$.
It is easy to see that there is a $\le^*$-unbounded family below $\emptyset'$: 
let $f_e(x)$ be $\varphi_e(x)$ unless there is some $y\le x$ where $\varphi_e(y)\dvg$, 
if there is then $f_e$ can take the value $0$ for all but finitely many inputs.

This notion of $\le^*$-unbounded families is one of various notions that 
take a property of a single noncomputable object and pull it apart into computable components.
In this case, the property for a single object is hyperimmunity.
A hyperimmune function is one that is not dominated by any computable function; 
a $\le^*$-unbounded family is one with members that are not uniformly dominated by any computable function.
Often in these cases, we can recover the single object property from the property for families:

\begin{remark}
    \label{hypimm<bd}
    For any $\langle f_e\rangle_{e\in\omega}\in\bd$, the map $x\mapsto\max_{e<x}f_e(x)$ is hyperimmune.
    Thus, $\hypimm\le_s\bd$.
\end{remark}

As we show in Corollary \ref{stricthi<b}, this Medvedev relation is strict.
This is not the case in general; we take a brief aside to show that every dominating function computes a $\le^*$-dominating (or cofinal) sequence.
Thus, the following mass problem analogue of the cardinal characteristic $\mathfrak d$ will not provide us with any unexpected results.

\begin{definition}[$\mathcal D$]
    A sequence $\langle f_e\rangle_{e\in\omega}$ of computable functions is called \emph{$\le^*$-dominating} if for any computable function $h$ there is
    some $e$ such that $h\le^* f_e$.
    Let $\mathcal D$ be the mass problem of $\le^*$-dominating sequences.
\end{definition}

The following result also provides a connection in this framework to the main focus of the work in \cite{LMNS:2023}.
It was shown that the mass problem of dominating functions $\domfcn$ is Medvedev equivalent to both the mass problem equivalents
for the ultrafilter number and the independence number.

\begin{proposition}
    \label{domfcn=d}
    $\domfcn\equiv_s\mathcal D$. That is, there are Turing functionals that compute from every $\le^*$-dominating family a dominating function,
    and vice versa.
\end{proposition}

\begin{proof}
    It is easy to see $\domfcn\le_s\mathcal D$, the argument is the same as Remark \ref{hypimm<bd}.
    If $\langle f_e\rangle_{e\in\omega}\in\mathcal D$ then the map $x\mapsto\max_{e<x}f_e(x)$ is dominating.

    \vspace*{1em}

    \noindent For the other reduction, let $h\in\domfcn$.
    For each $e$, define
        $$f_e(x)=\begin{cases}\varphi_{e,e+h(x)}(x)&\text{if }(\forall y\le x)\varphi_{e,e+h(x)}(y)\cnv,\\0&\text{otherwise.}\end{cases}$$
    Each $f_e$ is either $\varphi_e$ or takes the value $0$ on all but finitely many inputs, and so $f_e$ is always computable.

    Whenever $\varphi_e$ is a total function, the map $\psi_e:x\mapsto\mu s[\varphi_{e,s}(x)\cnv]$ is also total and computable, 
    and thus dominated by $h$.
    Hence, when $\varphi_e$ is total $\varphi_{e,h(x)}(x)$ converges on all but finitely many $x$, 
    so we can take the maximum of the stages of convergence for all other $x$ to find some $N$ where $\varphi_{e,N+h(x)}(x)$ always converges.
    There will be some index $e'>N$ for $\varphi_e$ and so $f_{e'}=\varphi_e$.
    Thus $\langle f_e\rangle_{e\in\omega}$ contains every computable function and so is $\le^*$-dominating.
\end{proof}

This argument does not work to compute $\le^*$-unbounded families from hyperimmune functions.
We rely on convergence on a cofinite set with a dominating function to ensure that each function of the constructed sequence is computable individually.
Attempting this with a hyperimmune function, we can only determine its behaviour on an infinite set.

It does turn out that a similar construction can be used to show that $\bd$ and $\af$ are Medvedev-equivalent in this setting.
Using the fact that each element of a $\le^*$-unbounded sequence is computable we can work to make each element of the MED family computable.
This result contrasts with the case in set theory; it is consistent with ZFC that $\mathfrak b<\mathfrak a_e$.
The argument relies on the computability setting and on the sequence structure of our families; there is no analogue in set theory.

\begin{proposition}
    \label{af=b}
    $\af\le_s\bd$.
    That is, there is a Turing functional that computes from every $\le^*$-unbounded family of functions a MED family.
\end{proposition}

\begin{proof}
    Let $\langle f_e\rangle_{e\in\omega}\in\mathcal B$.
    We will construct a family $g=\langle g_{e,k}\rangle_{e,k\in\omega}$ 
    where each $g_{e,k}$ is an attempt to make a computable function which is infinitely often equal to $\varphi_e$.

    For each $e,k$ define:
        $$g_{e,k}(x)=\begin{cases}\varphi_{e,f_k(x)}(x)&\text{if }(\forall n<\langle e,k\rangle)[\varphi_{e,f_k(x)}(x)\cnv\neq g_n(x)],\\
            1+\max_{n<\langle e,k\rangle}g_n(x)&\text{otherwise.}\end{cases}$$
    This uses only a single $f_k$ to compute its value, and hence $g_{e,k}$ is computable.
    At no point can $g_{e,k}(x)$ be defined to be equal to $g_n(x)$ for $n<\langle e,k\rangle$, so $g$ is ED.
    
    For each index $e$, if $\varphi_e$ is total then the function $\psi_e:x\mapsto\mu s[\varphi_{e,s}(x)\cnv]$ is also total and computable.
    Hence, there is some $k$ such that $f_k>^\infty\psi_e$.
    For the infinitely many $x$ where $f_k(x)>\psi_e(x)$, we have $\varphi_{e,f_k(x)}(x)\cnv$.
    For each such $x$, either $g_{e,k}(x)=\varphi_e(x)$ or there is $n<\langle e,k\rangle$ with $g_n(x)=\varphi_e(x)$.
    Thus, there must be some $n\le\langle e,k\rangle$ where $g_n=^\infty\varphi_e$, and so $g$ is MED.
\end{proof}

\begin{remark}
    Every MED family is an $\le^*$-unbounded family of functions.
    Thus, $\bd\equiv_s\af$
\end{remark}

\begin{proof}
    We see that $\bd\le_s\af$ holds by inclusion: if $h$ is computable and $g_e=^\infty h+1$ then ${h<^\infty g_e}$. 
    Thus, together with Proposition \ref{af=b} we obtain the equivalence ${\bd\equiv_s\af}$.
\end{proof}

The specification of towers using a \emph{descending} sequence of functions is not arbitrary in the computability case.
If we replace the above definition of towers with one for \emph{ascending} sequences of functions, 
then the mass problem will coincide in the Medvedev degrees with $\bd$ and $\af$:
Every maximal ascending tower is an $\le^*$-unbounded sequence. 
Given any $\le^*$-unbounded sequence, we can obtain a maximal ascending tower.
In fact, we can obtain some other nice properties which make this result useful to reference:

\begin{remark}
    \label{niceb}
    Every $\le^*$-unbounded sequence of functions computes an $\le^*$-unbounded sequence of 
    strictly increasing functions which is linearly \mbox{ordered by $<$}.
\end{remark}

\begin{proof}
    Let $\langle f_e\rangle_{e\in\omega}\in\bd$.
    For each $e$ we have $f_e'(x)=x+\max_{k<x}f_e(k)$ strictly increasing.
    The sequence $\langle g_e\rangle_{e\in\omega}$ defined by $g_e=\sum_{i<e}(1+f_i')$ is as required.
\end{proof}

The following result can be obtained from results in the next section; we give another proof via Proposition \ref{tf<t} and Theorem \ref{t<b}.
However, the tools used for a direct proof are helpful for working with descending towers of functions in general.
We include such a proof.

\begin{theorem}
    \label{tf<b}
    $\tf\le_s\bd$.
    That is, there is a Turing functional that from every $\le^*$-unbounded sequence of functions computes a maximal (descending) tower of functions.
\end{theorem}

In fact, this relation is strict in that there is a maximal tower which computes no $\le^*$-unbounded sequence.
This follows from Corollary \ref{septsbd} below.

Roughly speaking, we want to take a well-behaved ascending tower 
and map each of its constituent functions to something that behaves like an inverse.
The following definition is sufficient:

\begin{definition}
    \label{tildeop}
    Let $f\in\wtow$ be unbounded. Define $\tilde f\le_T f$ by
        $$\tilde f(x)=\mu n[f(n)>x].$$
\end{definition}

Unless $f$ is the identity function, this is not a true inverse, but it achieves the goal of mapping fast-growing functions to slow-growing ones.
We can observe that for nondecreasing functions, $\tilde\cdot$ is an involution:
    $$\smash{\tilde{\tilde f}}(x)=\mu n[\mu k[f(k)>n]>x]=\mu n[k\le x\to f(k)\le n]=\max_{k\le x}f(k)=:f'(x).$$

Now we can prove the result:

\begin{proof}[Proof (of Theorem \ref{tf<b})]
    Let $f=\langle f_e\rangle_{e\in\omega}\in\bd$.
    Assume without loss of generality that $f$ is in the nice form described in Remark \ref{niceb}.
    We claim that $g=\langle\tilde f_e\rangle_{e\in\omega}$ is a (descending) tower of functions.

    Fix some index $e$.
    For all $x$ we have 
        $$\tilde f_e(x)=\mu n[f_e(n)>x]\ge\mu n[f_{e+1}(n)>x]=\tilde f_{e+1}(x),$$ 
    as $f_{e+1}>f_e$.
    Hence $\langle\tilde f_e\rangle_{e\in\omega}$ is linearly ordered by $\ge$.

    We show that inequality holds on an unbounded and hence infinite set.
    Suppose we have $x$ such that $\tilde f_e(x)=\tilde f_{e+1}(x)=:k$.
    By the definition of $\tilde\cdot$ and the assumptions on $f$, we have 
        $$f_e(k-1)<f_{e+1}(k-1)\leq x<f_e(k)<f_{e+1}(k)$$
    and $k$ is least such that this holds. 
    Thus, we get $\tilde f_{e+1}(f_e(k))=k<\tilde f_e(f_e(k))$, giving us an inequality for an input greater than $x$.
    Hence, the set of points where inequality holds is unbounded.

    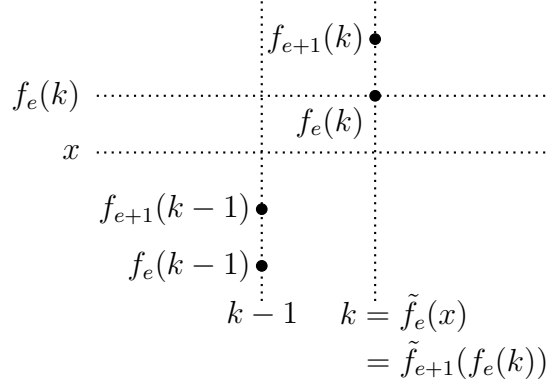
\begin{figure}[h]
        \centering
        \begin{tikzpicture}
            \draw[black,thick,dotted] (3,0) -- (-3,0) node[anchor=east]{$x$};
            \draw[black,thick,dotted] (3,0.75) -- (-3,0.75) node[anchor=east]{$f_e(k)$};
            \draw[black,thick,dotted] (-0.75,2) -- (-0.75,-2) node[anchor=west,xshift=-1.5em,yshift=-0.3em]{$k-1$}; 
            \draw[black,thick,dotted] (0.75,2) -- (0.75,-2) node[anchor=west,xshift=-1.5em,yshift=-0.3em]{$k=\tilde f_e(x)$}
                node[anchor=west,xshift=-0.65em,yshift=-1.8em]{$=\tilde f_{e+1}(f_e(k))$};
            \filldraw[black] (-0.75,-1.5) circle (2pt) node[anchor=east]{$f_e(k-1)$};
            \filldraw[black] (-0.75,-0.75) circle (2pt) node[anchor=east]{$f_{e+1}(k-1)$};
            \filldraw[black] (0.75,0.75) circle (2pt) node[anchor=north east]{$f_e(k)$};
            \filldraw[black] (0.75,1.5) circle (2pt) node[anchor=east]{$f_{e+1}(k)$};
        \end{tikzpicture}
        \caption{Illustration of the argument that $\langle\tilde f_e\rangle_{e\in\omega}$ is strictly ordered by $\ge^*$.}
    \end{figure}

    To show the tower $g$ is maximal, let $h$ be an arbitrary unbounded computable function.
    Fix $e$ such that the set $I=\{x:f_e(x)>\tilde h(x)\}$ is infinite.
    For $x\in I$ we have that $\tilde f_e(\tilde h(x))\le x<h'(\tilde h(x))$ by definitions.
    As $h'$ is the cumulative maximum of $h$, and $\tilde h(x)$ is always the least input to $h$ that crosses a threshold,
    we always have $h'(\tilde h(x))=h(\tilde h(x))$.
    Thus, for $x\in I$ we have $\tilde f_e(\tilde h(x))<h(\tilde h(x))$ and so $\tilde f_e<^\infty h$ as required.
\end{proof}

To verify that these notions are nontrivial, we show that no maximal tower of functions is computable:

\begin{proposition}
    Every $g=\langle g_e\rangle_{e\in\omega}\in\tf$ computes an unbounded function $p$ such that $p\le^*g_e$ for every $e$.
    By maximality of $g$, $p$ and thus $g$ are necessarily noncomputable.
\end{proposition}

\begin{proof}
    Let $g=\langle g_e\rangle_{e\in\omega}\in\tf$.
    We define $p$ to be unbounded on a sequence $\langle x_n\rangle_{n\in\omega}$, and $0$ elsewhere.
    Define $x_n$ recursively as follows:
    Let $x_0=0$.
    Define:
        $$x_{n+1}=\min\{x>x_n:g_{n+1}(x)>g_n(x_n)\text{ and }(\forall k\le n)[g_{n+1}(x)\le g_k(x)]\}.$$
    This is well defined as each function $g_e$ is unbounded, and for $k<e$ we have $g_e\le^*g_k$.
    Define:
        $$p(x)=\begin{cases}g_n(x)&\text{if }x=x_n,\\0&\text{otherwise.}\end{cases}$$
    The function $p$ is strictly increasing on the sequence $\langle x_n\rangle_{n\in\omega}$ and hence unbounded.
    For each $n$ and $x\ge x_n$ we have $p(x)\le g_n(x)$, and so $p\le^* g_n$.
\end{proof}

We take a moment to consider a natural property of such associated functions:

\begin{remark}
    If $\langle g_e\rangle_{e\in\omega}$ is a maximal tower, and $p$ is a function such that $p\le^*g_e$ for every $e$
    then $p$ does not dominate any unbounded computable function.
\end{remark}

\begin{proof}
    For every unbounded computable function $h$, there is an index $e$ such that ${p\le^*g_e<^\infty h}$.
\end{proof}

There would be some hope that we can study functions that do not dominate any unbounded computable function in order to find a lower bound
on the complexity of towers of functions.
Unfortunately, it is not so.

Generally, the easiest approach to constructing such a function is to ensure it has value 0 
on a set which meets every infinite computable set infinitely often, its behaviour elsewhere has no relevant effect.
In this way, we can observe that there is such a function in every noncomputable degree:

\begin{proposition}
    Suppose $Z$ is a noncomputable set. 
    Then there is an unbounded function $p\equiv_TZ$ which does not dominate any unbounded computable function.
\end{proposition}

\begin{proof}
    Define $I=\{Z\restriction n:n\in\omega\}$, where the initial segments of $Z$ are identified with natural numbers.
    The set $I$ is immune and $I\equiv_TZ$.
    Define:
        $$p(x)=\begin{cases}x&\text{if }x\in I,\\0&\text{otherwise.}\end{cases}$$
    The set $I$ is infinite, so $p$ is unbounded.
    If $h$ is an unbounded computable function, then the set $H=\{x:h(x)>0\}$ is infinite and computable, and so $H\smallsetminus I$ is infinite.
    For each $x\in H\smallsetminus I$ we have $p(x)=0<h(x)$ and so $p$ does not dominate $h$.
    We can recover $I$ as $\{x:p(x)>0\}$ and so their degrees coincide.
\end{proof}

In the case that such a function happens to be monotone, we can observe that then the function must have hyperimmune degree:

\begin{proposition}
    Let $p$ be an unbounded and monotone function that does not dominate any unbounded computable function.
    Then $p$ computes a hyperimmune function.
\end{proposition}

\begin{proof}
    Recall the $\tilde\cdot$ operator from Definition \ref{tildeop}.
    We show that $\tilde p$ is not hyperimmune-free. 
    Let $h$ be computable. If $h$ is bounded, we are done.
    In the unbounded case, the argument is analogous to the maximality argument in Theorem \ref{tf<b}, so we omit details:
    If $h$ is unbounded, then $\tilde h$ is well-defined and so the set $I=\{x:p(x)<\tilde h(x)\}$ is infinite.
    As both $p$ and $\tilde h$ are monotone, $\tilde h[I]$ is an infinite set of witnesses to $\tilde p(x)>^\infty h(x)$.
\end{proof}

\subsection{Medvedev relations between sequences of sets and of functions}
\label{seqsets}

Recall Definitions \ref{asdef} and \ref{tsdef} as given in \cite{LMNS:2023} for 
the mass problems $\as$ and $\ts$ of MAD families and maximal towers of sets, respectively.
Also, consider the following notions:

\begin{definition}[$\ts^-$]
    \label{t-def}
    A sequence of sets $\langle G_e\rangle_{e\in\omega}$ is called a \emph{weak tower} if it is linearly ordered by $\supseteq^*$.
    A weak tower is maximal if there is no infinite computable set $R$ such that $R\subseteq^*G_e$ for all $e$.
    Let $\ts^-$ be the mass problem of maximal weak towers.
\end{definition}

\begin{definition}[$\ps$]
    \label{psdef}
    Recall that a family has the \emph{strong finite intersection property} (SFIP) if any finite subfamily has infinite intersection.
    A family $\langle G_e\rangle_{e\in\omega}$ is said to have the \emph{pseudointersection property} if it has the SFIP and there is
    no infinite computable set $R$ such that $R\subseteq^*G_e$ for all $e$. 
    Let $\ps$ be the mass problem of families with the pseudointersection property.
\end{definition}

In the set theory case, there is no difference between strong and weak towers as we define them here, 
any weak tower in set theory yields a strong tower of at most the same cardinality by taking a representative from each $=^*$-equivalence class.
In computability theory, there is no obvious way to do this.
For cardinal characteristics, we have only $\mathfrak p=\mathfrak t\le\mathfrak a$, 
so weak towers may provide a closer correspondence to the set theory case.
It is not currently known whether or not $\ts^-\equiv_s\ts$, but there is no clear evidence for such a reduction.

The set theory argument of $\mathfrak p=\mathfrak t$ is the celebrated work of Malliaris and Shelah \cite{MS:2016}.
We obtain the analogous result, trivial in comparison:

\begin{remark}
    \label{t-=p}
    $\ts^-\equiv_s\ps$.
    That is, there are Turing functionals that from every maximal weak tower compute a family with the pseudointersection property
    and vice versa.
\end{remark}

\begin{proof}
    Every maximal weak tower has the pseudointersection property, as every finite subset has a $\subseteq^*$-minimal element.
    Hence, $\ts^-\le_s\ps$ by inclusion.

    If $\langle G_e\rangle_{e\in\omega}\in\ps$ then the sequence of partial intersections 
    $\langle\bigcap_{k\le e}G_k\rangle_{e\in\omega}$ is a weak maximal tower.
    Hence, $\ps\le_s\ts^-$.
\end{proof}

We will show the connections between the mass problems for families of sets and those for functions defined above.

\begin{proposition}
    \label{tf<t}
    $\tf\le_s\ts$.
    That is, there is a Turing functional that from every maximal tower of sets computes a maximal tower of functions.
\end{proposition}

\begin{proof}
    Let $\langle G_e\rangle_{e\in\omega}\in\ts$.
    Define the tower of functions $\langle g_e\rangle_{e\in\omega}$ by
        $$g_e(x)=\begin{cases}x&\text{if }x\in G_e,\\0&\text{otherwise.}\end{cases}$$

    As each set $G_e$ is infinite, each function $g_e$ is unbounded.
    We have ${g_{e+1}(x)>g_e(x)}$ if and only if $x\in G_{e+1}\smallsetminus G_e$, and so only finitely often.
    Hence ${g_{e+1}(x)\le^*g_e(x)}$.
    Similarly, $g_{e+1}<g_e$ if and only if $x\in G_e\smallsetminus G_{e+1}$ which is infinite, and so $g_{e+1}<^\infty g_e$.
    Hence $\langle g_e\rangle_{e\in\omega}$ is a tower of functions.

    To show maximality, let $h$ be an unbounded computable function. 
    Let ${H=\{x:h(x)>0\}}$.
    As $H$ is computable, we can fix $e$ such that $H\smallsetminus G_e$ is infinite.
    For each $x\in H\smallsetminus G_e$ we have $g_e(x)=0<h(x)$, and so $h$ is not dominated by $g_e$.
    Hence $\langle g_e\rangle_{e\in\omega}\in\tf$.
\end{proof}

Perhaps the most interesting and difficult question regarding the mass problems of sequences of functions defined in this section is whether
there is a Muchnik or Medvedev reduction between $\af$ and $\as$.
It is not clear that there should be any such relation. 
One can attempt to consider reductions that take functions to their graphs or sets to their characteristic functions and
such reductions do indeed preserve the disjointness requirements.
However, not every set is the graph of a function, and not every function is 2-valued or indeed has infinite support, 
and so the results will not be maximal.

In fact, it is true that $\as\le_s\af$, but we do not use the internal structure of the MED family in the reduction:

\begin{theorem}
    \label{t<b}
    $\ts\le_s\bd$.
    That is, there is a Turing functional that from every $\le^*$-unbounded family of functions computes a maximal tower of sets.
\end{theorem}

\begin{proof}
    Let $f=\langle f_e\rangle_{e\in\omega}$ be an $\le^*$-unbounded family of functions, with the properties specified in Remark \ref{niceb}.
    That is, each $f_e$ is strictly increasing, and $f$ is linearly ordered by $<$.
    We aim to further alter the family so that the sequence of ranges forms our maximal tower.
    We construct a family $g=\langle g_e\rangle_{e\in\omega}\le_Tf$ such that:
    \begin{enumerate}
        \item[(1)] For each $e$, $f_e\le g_e$, so $g$ is $\le^*$-unbounded,
        \item[(2)] for each $e$, $g_e$ is computable and increasing, so $\ran(g_e)$ is computable,
        \item[(3)] for each pair $e<k$, $\ran(g_k)\subseteq\ran(g_e)$ and $\ran(g_e)\smallsetminus\ran(g_k)$ is infinite.
    \end{enumerate}
    Unboundedness of $g$ will provide maximality for the tower of ranges.

    The idea is to make the ranges of each $g_e$ thinner than the previous, by skipping values that are too small to satisfy the above properties.
    We never run out of space to push the output further within the range of the previous $g_e$, as each is unbounded.

    \vspace*{1em}

    \noindent We define $g$ recursively. 
    Let $g_0=f_0$.
    For each $e$ and $x$, let
        $$N_{e,x}=\max\{f_{e+1}(x),\max_{y<x}g_{e+1}(y)\},$$
    and
        $$g_{e+1}(x)=\min\{y\in\ran(g_e):(\exists z\in\ran(g_e))[N_{e,x}<z<y]\}.$$
    Each function $f_e$ and previously defined function $g_e$ is unbounded, 
    so there are always infinitely many attained values above the maximum of a finite set, and $g$ is well defined.
    For each $e$ and $x$, we have $f_{e+1}(x)\le N_{e,x}<g_{e+1}(x)$, and so (1) holds.
    Similarly, for $y<x$ we have $g_{e+1}(y)<g_{e+1}(x)$ so $g_{e+1}$ is increasing. 
    As $g_{e+1}(x)$ is computed using a finite number of other computable functions, it is computable, and so (2) holds.
    By definition, $g_{e+1}$ always outputs members of the range of $g_e$, so $\ran(g_{e+1})\subseteq\ran(g_e)$.
    We also have that there is always some member of $\ran(g_e)$ strictly between $g_{e+1}(x)$ and $g_{e+1}(x+1)$,
    so infinitely many members of $\ran(g_e)$ are omitted, and (3) holds.
    Hence, $\langle\ran(g_e)\rangle_{e\in\omega}$ is a tower of sets.

    \begin{figure}[h]
        \centering
        \begin{tikzpicture}
            \draw[black,dotted] (5,0) -- (-5,0) node[anchor=east]{$g_e$};
            \draw[black,dotted] (5,-1) -- (-5,-1) node[anchor=east]{$f_{e+1}$};
            \draw[black,dotted] (5,-2) -- (-5,-2) node[anchor=east]{$g_{e+1}$};
            \foreach \x/\y in {-5/0,-4/0,-3/0,-2/0,-1/0,0/0,1/0,2/0,3/0,4/0,5/0,-5/-1,-2.7/-1,1.4/-1,-3/-2,-1/-2,3/-2}
                \filldraw[black] (\x,\y) circle (2pt);
            \foreach \x in {-4,-2,2}
                \draw[black,dotted] (\x,0) -- (\x,-2);
            \foreach \x in {-3,-1,3}
                \draw[black,thick,dotted] (\x,0) -- (\x,-2);
        \end{tikzpicture}
        \caption{Illustration of how the ranges of each $g_e$ are made to form a tower.}
        \label{rangetower}
    \end{figure}
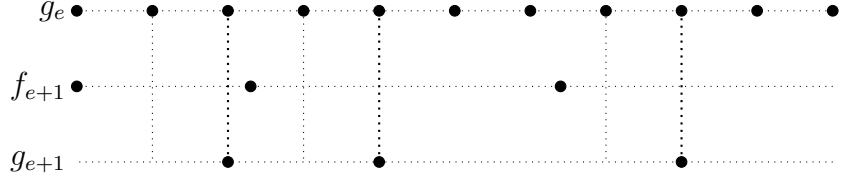

    In Figure \ref{rangetower}, we can see how, with each value included in the range of the next function, 
    an element of the previous range is left out.
    Also observe how $g_{e+1}>f_{e+1}$ so $\le^*$-unboundedness of the sequence is preserved.

    To show maximality, let $R$ be an infinite computable set, and let $p$ be the increasing computable function enumerating $R$.
    By $\le^*$-unboundedness of $g$, fix $e$ such that there are infinitely many $x$ with $p(2x)<g_e(x)$.
    For these $x$ we have:
    \begin{align*}
        |\!\ran(g_e)\cap[0,p(2x)]|&<x,\\
        |\!\ran(p)\cap[0,p(2x)]|&=2x.
    \end{align*}
    As $R=\ran(p)$, for an infinite set of $x$ we thus have that $|R\smallsetminus\ran(g_e)|>x$, 
    and so $R\smallsetminus\ran(g_e)$ is infinite.
    Thus, $\langle\ran(g_e)\rangle_{e\in\omega}\in\ts$.
\end{proof}

\subsection{Separating mass problems}
\label{sepmp}

It can be difficult to show that a mass problem is not reducible to another.
One has to show that there is a member of one problem that does not compute any member of the other.
If we have an example of a non-reduction between two well-studied classes, then it may be useful:

\begin{proposition}[\cite{LMNS:2023}, Thm.\ 3.1]
    \label{lmns3.1}
    $\ts\le_s\nonlow$.
    That is, there is a Turing functional that from every non-low set computes a maximal tower of sets.
    \qed
\end{proposition}

We know that every $\le^*$-unbounded sequence of functions computes a hyperimmune function, 
so to show that $\ts\ngeq_w\bd$ it suffices to show there is a non-low hyperimmune-free set.

\begin{remark}
    \label{hifnl}
    $\hypimm\nleq_w\nonlow$.
    That is, there exists a non-low set of hyperimmune-free degree.
\end{remark}

\begin{proof}
    Every noncomputable low set has hyperimmune degree, so any noncomputable hyperimmune-free set must be non-low.
    Thus, applying the basis theorem for hyperimmune-free sets (e.g.\ \cite[Thm.\ 1.8.42]{Nies:2008}) 
    to any $\Pi^0_1$-class without computable members yields a non-low hyperimmune-free set.
\end{proof}

\begin{corollary}
    \label{septsbd}
    There is a maximal tower of sets which does not compute any $\le^*$-unbounded sequence of functions.
    Thus, $\ts\ngeq_w\bd$ and hence $\ts<_s\bd$.
\end{corollary}

\begin{proof}
    Take a hyperimmune-free non-low set, as in Proposition \ref{hifnl}.
    The maximal tower of sets $G$ it computes by Proposition \ref{lmns3.1} cannot compute any hyperimmune function, 
    but every $\le^*$-unbounded family of functions computes a hyperimmune function by Remark \ref{hypimm<bd}.
    Hence, $G$ does not compute any $\le^*$-unbounded family of functions.
\end{proof}

While we know that no member of mass problems defined as in Sections \ref{analoguesaftf} and \ref{seqsets} is computable, 
we would also like to know if there is some such class with members in every noncomputable degree.
This is not the case.
The following definition gives us a lower bound for the complexity of any such mass problem.

\begin{definition}[\cite{LMNS:2023}, Def.\ 3.3]
    An oracle $Z$ is called \emph{index guessable} if $\emptyset'$ computes from $k$ an index for $\Phi^Z_k$ whenever $\Phi^Z_k$ is computable.
    That is, there is a Turing functional $\Gamma$ such that
        $$\Phi^Z_k\text{ is computable}\implies\Phi^Z_k=\varphi_{\Gamma^{\emptyset'}\!(k)}.$$
    The functional $\Gamma$ does not need to be total.
\end{definition}

Of course, we need to verify that this notion is nontrivial for it to be useful.

\begin{proposition}[\cite{LMNS:2023}, Prop.\ 2.4]
    \label{lmns2.4}
    Every $\Delta^0_2$ 1-generic set is index guessable.
    In particular, there is a noncomputable index guessable set.
    \qed
\end{proposition}

It was shown in \cite[Prop.\ 2.4]{LMNS:2023} that no index guessable set computes an element of $\ts$.
This is the strongest result they need, as $\ts$ is the Muchnik-least mass problem they consider.
We adapt the argument to be more general, so as to suit any analogues of cardinal characteristics in this framework:

\begin{proposition}
    \label{iggeneral}
    Let $P$ be a property of sequences of computable functions such that if $f=\langle f_e\rangle_{e\in\omega}$ has property $P$, then:
    \begin{enumerate}
        \item $f$ is noncomputable,
        \item if $\hat f=\langle \hat f_e\rangle_{e\in\omega}$ and $\hat f_e=^*f_e$ for each $e$, then $\hat f$ has property $P$.
    \end{enumerate}
    Then no index guessable oracle computes a sequence with property $P$.
\end{proposition}

\begin{proof}
    Let $Z$ be index guessable.
    With a view to contradition, suppose that $f=\langle f_e\rangle_{e\in\omega}$ has property $P$, and $f\le_TZ$ via $\Phi_k$.
    Using the index $k$, we can find a computable function $g$ such that $\Phi^Z_{g(e)}=f_e$.
    Thus, there is a Turing functional $\Gamma$ such that $f_e=\Phi^Z_{g(e)}=\varphi_{\Gamma^{\emptyset'}\!(g(e))}$ 
    and thus by the Limit Lemma there is a computable function $h$ such that $f_e=\varphi_{\lim_sh(e,s)}$.
    Define $\hat f=\langle\hat f_e\rangle_{e\in\omega}$ as follows:
    Fix $e$ and $x$.
    Let $s>x$ be least such that $\varphi_{h(e,s),s}(x)\!\!\downarrow$.
    Define $\hat f_e(x)=\varphi_{h(e,s),s}(x)$.
    As for a fixed $e$, there are only finitely many $s$ where $\varphi_{h(e,s)}\neq f_e$, we have that $\hat f_e=^*f_e$.
    Thus, $\hat f$ would have property $P$ but be computable.
\end{proof}

This also yields the promised separation between $\hypimm$ and $\bd$:

\begin{corollary}
    \label{stricthi<b}
    There exists a hyperimmune function which computes no $\le^*$-unbounded family of functions.
    Thus, $\hypimm\ngeq_w\bd$ and hence $\hypimm<_s\bd$.
\end{corollary}

\begin{proof}
    We saw in Remark \ref{hypimm<bd} that $\hypimm\le_s\bd$.

    By Proposition \ref{lmns2.4} there is a noncomputable $\Delta^0_2$ index guessable set.
    Every noncomputable $\Delta^0_2$ degree is hyperimmune, 
    and by Proposition \ref{iggeneral} it must not compute a $\le^*$-unbounded sequence of functions.
\end{proof}

\subsection{Descriptive complexity of MED families} 

In both the set theory and computability settings, the question of the descriptive complexity of these families proves interesting.
As we know, there are $\Delta^0_2$ members for each of our mass problem analogues of cardinal characteristics, 
the immediate question is whether they can be c.e.\ or co-c.e.
Maximal towers of sets must not be c.e., as per \cite[Fact 2.3]{LMNS:2023}.
The argument requires no alteration to show that no maximal weak tower of sets can be c.e.\ either.
We provide the argument they outline:

\begin{remark}
    If $G=\langle G_e\rangle_{e\in\omega}$ is a maximal weak tower of sets, then $G$ is not c.e.
\end{remark}

\begin{proof}
    Let $x_0$ be the first element enumerated into $G_0$.
    For each $n$, let $x_{n+1}$ be the first element enumerated into $\bigcap_{i\le n+1}G_i$ which is greater than $x_n$.
    Then, $X=\{x_i:i\in\omega\}\subseteq^*G_i$ for each $i$ and $X$ is enumerated in increasing order and is hence computable.
\end{proof}

We can give a similar argument for towers of functions, 
showing that no maximal tower of functions can be approximated from below in the same manner as c.e.\ sets.
We hence define c.e.\ functions as follows:

\begin{definition}
    \label{cefunc}
    A function $f$ is called \emph{computably enumerable (c.e.)}\ if there is a uniformly computable sequence $\langle f_n\rangle_{n\in\omega}$
    such that $f_0\equiv0$, $f=\lim_nf_n$, and for every $x$ there is at most one $n$ such that $f_n(x)\neq f_{n+1}(x)$.
\end{definition}

\begin{proposition}
    If $g=\langle g_e\rangle_{e\in\omega}$ is a maximal tower of functions, then $g$ is not c.e.
\end{proposition}

\begin{proof}
    Suppose that $g$ is c.e.
    Note that the map $\langle e,x,s\rangle\mapsto g_{e,s}(x)$ is computable and total.
    Define uniformly computable sequences $\langle s_n\rangle_{n\in\omega}$ and $\langle x_n\rangle_{n\in\omega}$ as follows:
    Let $\langle s_0,x_0\rangle=\mu\langle s,x\rangle[g_{e,s}(x)>0]$.
    For each $n$, define:
        $$\langle s_{n+1},x_{n+1}\rangle=\mu\langle s,x\rangle[g_{n+1,s}(x)>n+1\text{ and }x>x_n].$$
    The sequence of $x_n$ is increasing, so the sequence $\langle x_n\rangle_{n\in\omega}$ is computable.
    Thus, the function
        $$p(x)=\begin{cases}g_{n,s_n}(x)&\text{if }x=x_n,\\0&\text{if }x\notin\{x_n\}_{n\in\omega}\end{cases}$$
    is computable but dominated by every $g_e$.
\end{proof}

The Muchnik equivalence between $\as$ and $\ts$ maps co-c.e.\ sets to c.e.\ ones and vice versa, hence there cannot be any co-c.e.\ MAD family.
In \cite{LMNS:2023}, it is shown that there is a c.e.\ MAD family, and in fact every noncomputable c.e.\ degree computes one:

\begin{theorem}[\cite{LMNS:2023}, Thm.\ 5.1]
    \label{lmns5.1}
    For each noncomputable c.e.\ set $A$, there is a c.e. MAD family $G\le_TA$.
    \qed
\end{theorem}

In particular, there is a low MAD family, and so we get the strict inequality $\as<_s\nonlow$.
This also shows that there is a co-c.e.\ maximal tower.

The notion of c.e.\ functions in Definition \ref{cefunc} does not make sense to use for an MED family of functions
as each pair of members can only have common value $0$ in finitely many places.
In its place, we can consider a uniformly computable family of partial computable functions that each extend to a total computable function.
To get an analogous result to Theorem \ref{lmns5.1}, we identify partial functions with their graphs.

Say a family of partial computable functions $\langle\vartheta_e\rangle_{e\in\omega}$ is MED if $\vartheta_e\cap\vartheta_k$ is finite for all $e,k$
and for any computable function $h$ there is $e$ such that $h=^\infty\vartheta_e$.
We obtain the following as an analogue of Theorem \ref{lmns5.1}.

\begin{theorem}
    For each noncomputable c.e.\ set $A$ there is a MED family of partial computable functions ${\vartheta=\langle\vartheta_e\rangle_{e\in\omega}}$ 
    such that each $\vartheta_e$ extends to a total computable function, and the graph of $\vartheta$ is c.e.\ and below $A$.
\end{theorem}

\begin{proof}
    Let $\langle\varphi_e\rangle_{e\in\omega}$ be a uniform enumeration of partial computable functions built up in stages such that 
    if $\varphi_e(x)$ converges at stage $s$ then it does not change value at any stage after $s$.
    This reflects c.e.\ sets having computable approximations from below.

    Recall that we identify functions with their graphs.
    We consider the following requirements:
        $$P_{e,k}\!\!:\text{If }\varphi_e\smallsetminus\bigcup_{i<\langle e,k\rangle}\vartheta_i\text{ is infinite, then }
            |\vartheta_e\cap\varphi_e|>k.$$ 
    We say that $P_{e,k}$ is \emph{permanently satisfied} at stage $s$ if $|\vartheta_{e,s}\cap\varphi_{e,s}|>k$.

    \vspace*{1em}

    \noindent\emph{Construction.}

    \emph{Stage $s>0$:}
    See if there is $n:=\langle e,k\rangle<s$ such that $P_{e,k}$ is not permanently satisfied and there is $x$ such that
    \begin{enumerate}
        \item[1)] $\varphi_{e,s}(x)\!\downarrow$ and $(\forall i<\langle e,k\rangle)[\varphi_{e,s}(x)\neq\vartheta_{i,s-1}(x)]$,
        \item[2)] if $\vartheta_{e,s-1}(y)\!\downarrow$ then $y<x$,
        \item[3)] $A_s\!\!\restriction\!x\neq A_{s-1}\!\!\restriction\!x.$
    \end{enumerate}
    If found, choose $n$ least and then $x$ least that meet the above conditions (1)--(3).
    Define $\vartheta_{e,s}(x)=\varphi_{e,s}(x)$.

    \vspace*{1em}

    \noindent\emph{Verification.}\nopagebreak

    Each $\vartheta_e$ has its graph enumerated in increasing order of inputs, and hence can be extended to a total computable function.
    The graph of the universal function $\vartheta$ is enumerated as permitted by $A$, and so $\vartheta\le_TA$.

    \begin{claim}
        Each requirement $P_{e,k}$ is satisfied.
    \end{claim}

    Suppose there are $e,k$ such that $P_{e,k}$ is not satisfied by the construction.
    We achieve a contradiction by a permitting argument.
    Choose $\langle e,k\rangle$ to be the least example of $P_{e,k}$ failing, suppose at stage $s$ that $|\vartheta_{e,s}\cap\varphi_{e,s}|=k$.
    Its assumption holds, so $\varphi_e\smallsetminus\bigcup_{i<\langle e,k\rangle}\vartheta_i$ is infinite, 
    and so there are infinitely many $x$ that will eventually satisfy (1).
    All but finitely many of those $x$ will also always satisfy (2) as $\vartheta_e$ settles by stage $s$.
    Thus, there is a c.e.\ sequence $x_1<x_2<\cdots$ of elements which simultaneously satisfy (1) and (2) at all stages above $s$,
    and a corresponding c.e.\ sequence $s_1<s_2<\cdots$ of stages such that for all $i$, we have $\varphi_{e,s_i}(x_i)\!\downarrow$, 
    and no element $y\le x_i$ enters $A$ at a stage $t\ge s_i$.
    Otherwise all conditions \mbox{(1)--(3)} would have been met by $x_i$ simultaneously for sufficiently large stages 
    and $\vartheta_e(x_i)$ would converge past stage $s$.
    Define $B_e(y)=A_{s_i}(y)$ for all $y\le x_i$; $B_e$ is computable.
    As each $A\!\restriction\!x_i$ settles by stage $s_i$, we have that $B_e=A$, contradicting the noncomputability of $A$.
    Hence $P_{e,k}$ indeed was satisfied.

    \begin{claim}
        $\vartheta$ is MED.
    \end{claim}

    Let $e<m$, we claim that $|\vartheta_e\cap\vartheta_m|\le m$.
    Suppose $\vartheta_e(x)=\vartheta_m(x)$.
    If $\vartheta_e(x)$ converges first at stage $s$ due to the action of $P_n$, 
    then $\vartheta_{m,s}(x)\!\downarrow$ as $e<\langle m,k\rangle$ for every $k$ 
    so (1) would forbid $\vartheta_m(x)$ from being altered after stage $s$.
    Thus, also by (1), $n\le m$ as otherwise $P_n$ cannot have acted to define $\vartheta_{e,s}(x)$.
    Each $P_n$ only acts to define at most a single value of $\vartheta_e$, so there are at most $m$ many such $x$ with $\vartheta_e(x)=\vartheta_m(x)$.
    Hence, $\vartheta$ is ED.

    For each $e$ such that $\varphi_e$ is total, either $\vartheta_e\cap\varphi_e$ is infinite as the hypothesis of each $P_{e,k}$ is met, 
    or there is $k$ such that the hypothesis for $P_{e,k}$ is not met.
    If so, then there is $i<\langle e,k\rangle$ such that $\vartheta_i\cap\varphi_e$ is infinite.
    Hence, $\vartheta$ is MED.
\end{proof}

    \newpage
    \section{Mass Problems and Cicho\'n's Diagram}
\label{mpcichon}

The abstraction to Weihrauch problems under various names has been an important part of creating a framework for working with cardinal characteristics.
Greenberg, Kuyper, and Turetsky \cite{GKT:2019} show that this framework serves equally well for the analogues studied in computability, 
with an effectivity condition placed on reductions between problems.
The main focus of their work is on reconstructing the classical results about Cicho\'n's diagram from set theory 
and the results by Brendle et al.~\cite{BBNN:2015} in computability, using the framework of Weihrauch problems as an abstraction of both.

To consider effective analogues of much of Cicho\'n's diagram, we need codes for meagre and null sets.
It is sufficient to consider codes for $\mathbf\Sigma^0_2$ meagre sets and $\mathbf\Pi^0_2$ null sets.
Recall that we identify sequences of sets or functions with a single object via Definition \ref{seqcode}, 
and trees can be identified with sets by identifying strings with natural numbers.
We can code a closed nowhere dense set $C$ by a nowhere dense tree $T$ where $C=[T]$.
Thus, a code for a meagre set $M$ is a sequence $\langle T_n\rangle_{n\in\omega}$ of nowhere dense trees such that $M=\bigcup_n[T_n]$.
Similarly, we code an open set $U$ by a set whose complement is a code for a closed set.
A code for a null set $N$ is a pair of sequences $\langle U_n\rangle_{n\in\omega}$ of coded open sets and $\langle x_n\rangle_{n\in\omega}$
such that $N=\bigcap_nU_n$, $x_n$ is the measure $\lambda(U_n)$, and $x_n\le2^{-n}$.

Again, for the following definitions from \cite{GKT:2019} $\mathcal I$ is a non-principal ideal on an infinite set $X$.

\begin{definition}
    Define $\opn{Capture}(\mathcal I)$ as the problem of finding an element of $\mathcal I$ containing a given element of $X$.
    That is, $\opn{Capture}(\mathcal I)=\langle X,\mathcal I,\in\rangle$.
\end{definition}

\begin{definition}
    The problem $\opn{Pass}(\mathcal I)$ is the dual of $\opn{Capture}(\mathcal I)$,
    the problem of finding an element of $X$ outside of a given element of $\mathcal I$.
    That is, $\opn{Pass}(\mathcal I)=\langle\mathcal I,X,\notni\rangle$.
\end{definition}

\begin{definition}
    Define $\opn{Supset}(\mathcal I)$ as the problem of finding an element of $\mathcal I$ containing a given other element of $\mathcal I$.
    That is, $\opn{Supset}(\mathcal I)=\langle\mathcal I,\mathcal I,\subseteq\rangle$.
\end{definition}

\begin{definition}
    The problem $\opn{Spill}(\mathcal I)$ is the dual of $\opn{Supset}(\mathcal I)$,
    the problem of finding an element of $\mathcal I$ which is not contained in a given element of $X$.
    That is, $\opn{Spill}(\mathcal I)=\langle\mathcal I,\mathcal I,\not\supseteq\rangle$.
\end{definition}

We will primarily focus on the ideals $\mathcal M$ of $\mathbf\Sigma^0_2$ meagre subsets of $\wtow$ 
and $\mathcal N$ of $\mathbf\Pi^0_2$ null subsets of $\ttow$.

The aim of this section is to survey and make use of the results in \cite{GKT:2019}
to construct an analogue of Cicho\'n's diagram in computability with the setting of mass problems.
The constructions of morphisms are all from their work; we present them with more detail for clarity and verify any additional properties we need.

\subsection{Computable morphisms}

We build on the framework of \cite{GKT:2019} and associate a mass problem to each Weihrauch problem.
Many results analogous to those for cardinal characteristics or nonlowness classes can be obtained in this way.

\begin{definition}
    Let $A=\langle A_\inst,A_\sol,A\rangle$ be a Weihrauch problem.
    Recall that $A_\inst^r$ and $A_\sol^r$ are the sets of computable instances and solutions, respectively.
    Define $$\massp(A)
    =\{\langle f_e\rangle_{e\in\omega}:\forall e[f_e\in A_\sol^r]\text{ and }(\forall a\in A_\inst^r)\exists e[aAf_e]\}.$$
    This is the collection of complete solution sets, interpreted as sequences coded by a single object,
    restricting to computable instances and solutions.
    Let $\massp^Z(A)$ be the relativisation to $Z$.
\end{definition}

Recall the Weihrauch problems corresponding to $\mathfrak d$ and $\mathfrak b$:
$\opn{Dom}=\langle\wtow,\wtow,\le^*\rangle$ and $\opn{Esc}=\langle\wtow,\wtow,<^\infty\rangle$.
We have that $\massp(\opn{Dom})=\mathcal D$ and $\massp(\opn{Esc})=\mathcal B$.

To make use of this framework to obtain results for the setting of mass problems, 
we observe that whenever a morphism is computable, we obtain the corresponding Medvedev reduction between the associated mass problems:

\begin{proposition}
    \label{morphismmedvedev}
    If $A\to B$ via a morphism $\varphi_\inst,\varphi_\sol$ such that there are (not necessarily total) Turing functionals
    $\Phi_\inst$ and $\Phi_\sol$ such that $\varphi_\inst(a)=\Phi_\inst^a$ and 
    $\varphi_\sol(b)=\Phi_\sol^b$ for every $a\in A_\inst^r$ and $b\in B_\sol^r$, then $\massp(A)\le_s\massp(B)$.
    We also get that $\massp(B^\bot)\le_s\massp(A^\bot)$.
\end{proposition}

\begin{proof}
    Let $\Gamma$ be the Turing functional where if $b=\langle b_e\rangle_{e\in\omega}\in\massp(B)$ then 
        $$\Gamma^b(e,x)=\Phi^{b_e}_\sol(x)=\varphi_\sol(b_e)(x),$$
    and let 
        $$\langle f_e\rangle_{e\in\omega}=\Gamma^b=\langle x\mapsto\Gamma^b(e,x)\rangle_{e\in\omega}.$$
    Now for any $a\in A_\inst^r$, $\Phi_\inst^a\in B_\sol^r$ so there is $e$ with $\Phi_\inst^aBb_e$ and so $aA\Phi_\sol^{b_e}$, that is, $aAf_e$.
    Hence $\langle f_e\rangle_{e\in\omega}\in\massp(A)$.
\end{proof}

The additional requirement on the computability of the morphism is necessary.
Without it, we cannot guarantee that $\langle f_e\rangle\le_T\langle b_e\rangle$ so there may not be a Muchnik/Medvedev reduction.
Almost all of the morphisms demonstrated in \cite{GKT:2019} are in fact computable, 
we will discuss those which are not in Section \ref{meagrenull}.

The following easy example is instructive:

\begin{remark}
    The morphism $\opn{Esc}\to\opn{Dom}$ is computable.
    Hence, $\mathcal B\le_s\mathcal D$.
\end{remark}

\begin{proof}
    Let $\Phi_\inst$ be the identity.
    Let $\Phi_\sol$ map a function $f$ to $f+1$.
    If $h$ is dominated by $f$, then $f+1$ is infinitely often greater than $h$.
    
    Thus, each $\le^*$-dominating sequence $\langle f_e\rangle_{e\in\omega}$ computes a $\le^*$-unbounded sequence $\langle f_e+1\rangle_{e\in\omega}$.
\end{proof}

This case is also interesting in that it is self-dual.
There is no general reduction between a Weihrauch problem and its dual.

\subsection{Computably coded meagre sets}
\label{meagrenull}

We define mass problem analogues of the cardinals associated with the meagre ideal using their respective Weihrauch problems.

\begin{definition}[$\opn{Add}(\mathcal M)$]
    Define $\opn{Add}(\mathcal M)=\massp(\opn{Spill}(\mathcal M))$.
    That is, the mass problem $\opn{Add}(\mathcal M)$ is the class of sequences $\langle M_e\rangle_{e\in\omega}$ of computable meagre sets
    such that for any computable meagre set $\tilde M$ there is $e$ such that $M_e\nsubseteq\tilde M$.
\end{definition}

\begin{definition}[$\opn{Cov}(\mathcal M)$]
    Define $\opn{Cov}(\mathcal M)=\massp(\opn{Capture}(\mathcal M))$.
    That is, the mass problem $\opn{Cov}(\mathcal M)$ is the class of sequences $\langle M_e\rangle_{e\in\omega}$ of computable meagre sets
    such that for every computable function $f$ there is $e$ such that $f\in M_e$.
\end{definition}

\begin{definition}[$\opn{Non}(\mathcal M)$]
    Define $\opn{Non}(\mathcal M)=\massp(\opn{Pass}(\mathcal M))$.
    That is, the mass problem $\opn{Non}(\mathcal M)$ is the class of sequences $\langle f_e\rangle_{e\in\omega}$ of computable functions
    such that for any computable meagre set $\tilde M$ there is $e$ such that $f_e\notin\tilde M$.
\end{definition}

\begin{definition}[$\opn{Cof}(\mathcal M)$]
    Define $\opn{Cof}(\mathcal M)=\massp(\opn{Supset}(\mathcal M))$.
    That is, the mass problem $\opn{Cof}(\mathcal M)$ is the class of sequences $\langle M_e\rangle_{e\in\omega}$ of computable meagre sets
    such that for any computable meagre set $\tilde M$ there is $e$ such that $\tilde M\subseteq M_e$.
\end{definition}

The morphisms defined in \cite{GKT:2019} between these Weihrauch problems are all computable.
We present them below with additional detail and highlight this fact.
In each case, we note that Proposition \ref{morphismmedvedev} gives the corresponding and dual Medvedev reductions.

\begin{proposition}
    The morphism $\opn{Pass}(\mathcal M)\to\opn{Supset}(\mathcal M)$ is computable.
    Hence, $\opn{Non}(\mathcal M)\le_s\opn{Cof}(\mathcal M)$ and $\opn{Add}(\mathcal M)\le_s\opn{Cov}(\mathcal M)$.
\end{proposition}

\begin{proof}
    Let $\Phi_\inst$ be the identity.
    Let $\Phi_\sol$ map a code for a meagre set $M$ to a function $f\notin M$:
    If $M=\langle T_n\rangle_{n\in\omega}$ is a code for a meagre set we let ${\sigma_0=\min({}^{<\omega}\omega\smallsetminus T_0)}$
    and $\sigma_{n+1}=\min([\sigma_n]^\prec\smallsetminus T_{n+1})$ for $n\ge0$.
    Then, $\Phi_\sol^M=\bigcup_n\sigma_n\notin M$.
    This is well defined as long as each tree $T_n$ is nowhere dense.

    If $M$ and $B$ are meagre sets, and $\Phi_\inst^M\subseteq B$, then $\Phi_\sol^B\notin B$ and hence is not \mbox{in $M$}.
\end{proof}

The morphisms that relate these to the problem $\opn{Dom}$ are also computable, giving us the picture regarding Baire space:

\begin{proposition}
    \label{CapMtoDom}
    The morphism $\opn{Capture}(\mathcal M)\to\opn{Dom}$ is computable.
    Hence, $\opn{Cov}(\mathcal M)\le_s\mathcal D$ and $\mathcal B\le_s\opn{Non}(\mathcal M)$.
\end{proposition}

\begin{proof}
    Let $\Phi_\inst$ be the identity.
    Let $\Phi_\sol$ map a function $g$ to the meagre set of functions $f\le^*g$.
    We obtain the code for such as a union of trees $T_n$, ${\Phi_\sol^g=\langle T_n\rangle_{n\in\omega}}$, where 
        $$T_n=\{\sigma:\forall k[n\le k<|\sigma|\rightarrow\sigma(k)\le g(k)]\}.$$ 
    If $f$ and $g$ are computable functions, and $\Phi_\inst^f=f\le^* g$, then $f\in\Phi_\sol^g$.
\end{proof}

\begin{proposition}
    The morphism $\opn{Dom}\to\opn{Supset}(\mathcal M)$ is computable.
    Hence, $\mathcal D\le_s\opn{Cof}(\mathcal M)$ and $\opn{Add}(\mathcal M)\le_s\mathcal B$.
\end{proposition}

\begin{proof}
    Let $\Phi_\inst$ map a function $f$ to the meagre set of all $f$-dominated functions, as in Proposition \ref{CapMtoDom}.
    For $\Phi_\sol$, we map a meagre set $M$ to a function $g$ such that if $M$ contains all $f$-dominated functions, then $f\le^* g$.
    Let $M=\langle T_n\rangle_{n\in\omega}$ be a code for a meagre set, and assume without loss of generality that $T_0\subseteq T_1\subseteq\cdots$.
    For each $n,k$, define the finite set
        $$S_{n,k}=\{\tau\in T_n:|\tau|=k\land \max\opn{range}\tau\le k\},$$
    and recursively find $\sigma_{n,k}$ so that $(S_{n,k}\hatc\sigma_{n,k})\cap T_n=\emptyset$.

    Define $g(0)=0$ and $g(n+1)=\max\opn{range}\sigma_{n+1,g(n)}$.
    Let $\Phi_\sol^M=g$.
    We can define the intervals
        $$J_n=[g(n),g(n)+|\sigma_{n,g(n)}|].$$

    Now, suppose towards contradiction that there is a (non-decreasing) function $f$ such that every $f$-dominated function is contained in $M$, 
    and $g$ does not dominate $f$.
    There is then an infinite set $X$ such that $g\restriction X\le f\restriction X$.
    We can assume that $X$ is such that the intervals $J_n$ are pairwise disjoint for $n\in X$.
    Define $x$ such that $x\restriction J_n=\sigma_{n,g(n)}$ for each $n\in X$, and $x(n)=0$ for $n\notin\bigcup_{n\in X}J_n$.
    For each $n\in X$ we have $x\notin T_n$ by the definition of $\sigma_{n,g(n)}$, and so $x\notin M$.

    For $n\in X$ and $i\in J_n$ we have that $x(i)\in\opn{range}(\sigma_{n,g(n)})$ so $x(i)\le g(n)\le f(n)$, 
    but $i\in J_n$ and so $i\ge g(n)\ge n$ and $f$ is non-decreasing so $f(n)\le f(i)$.
    Thus, $x\le f$ and so $x\in M$.

    If $f$ is a computable function and $M$ is a meagre set, and $\Phi_\inst^f\subseteq M$, then $M$ contains all $f$-dominated functions and so
    $f\le^*\Phi_\sol^M$.
\end{proof}

\subsection{Computably coded null sets}

As with meagre sets above, the morphisms between the Weihrauch problems for the null ideal are computable:

\begin{proposition}
    The morphism $\opn{Capture}(\mathcal N)\to\opn{Supset}(\mathcal N)$ is computable.
    Hence, $\opn{Cov}(\mathcal N)\le_s\opn{Cof}(\mathcal N)$ and $\opn{Add}(\mathcal N)\le_s\opn{Non}(\mathcal N)$.
\end{proposition}

\begin{proof}
    Let $\Phi_\inst$ map a point $x$ to the null set $\{x\}$, coded as 
    $\langle [x\!\!\restriction\!\!n]^\prec\rangle_{n\in\omega}$.
    The respective sequence of measures is just $\langle 2^{-n}\rangle_{n\in\omega}$.
    Let $\Phi_\sol$ be the identity.

    If $x$ is computable and $N$ is a null set, and $\Phi_\inst^x=\{x\}\subseteq N$, then we have that ${x\in N=\Phi_\sol^N}$.
\end{proof}

\begin{proposition}
    The morphism $\opn{Pass}(\mathcal N)\to\opn{Supset}(\mathcal N)$ is computable.
    Hence, $\opn{Non}(\mathcal N)\le_s\opn{Cof}(\mathcal N)$ and $\opn{Add}(\mathcal N)\le_s\opn{Cov}(\mathcal N)$.
\end{proposition}

\begin{proof}
    Let $\Phi_\inst$ be the identity.
    Let $\Phi_\sol$ map a null set $N=\bigcap_nU_n$ to a point $y\notin N$.
    We achieve this by finding $y\notin U_1$, which can be done as $\lambda(U_1)\le\frac12$.
    For each finite string $\sigma$ we have that $\lambda([\sigma]\smallsetminus U_1)$ is computable from $\lambda(U_1\smallsetminus[\sigma])$
    which can be uniformly computed from $U_1$ and $\lambda(U_1)$:
    it can be approximated from below by using the coding $U_1=\bigcup_n[\sigma_n]$ and 
    observing that the nature of each $[\sigma_n]\smallsetminus[\sigma]$ is easy to determine.
    From above we use $\lambda(U_1)$ and $U_1\cap[\sigma]$ in the same way.
    Thus, we can always find the next bit to make $\lambda([y\restriction n]\smallsetminus U_1)>0$ and in this way construct all of $y$.

    If $N$ and $\tilde N$ are null sets, and $\Phi_\inst^N=N\subseteq\tilde N$, then $\Phi_\sol^{\tilde N}$ is not in $\tilde N$ and so it is 
    also not in $N$.
\end{proof}

Connecting these to the previous Weihrauch problems is more involved.
In order to relate the ideals of meagre and null sets, it will be much easier to bridge the difference in underlying spaces first.
Let $\mathcal M_C$ be the ideal of $\mathbf\Sigma^0_2$ meagre subsets of $\ttow$.
The map $\rho:\wtow\to\ttow$ is defined \cite{GKT:2019} by
    $$\rho(f)=0^{f(0)}10^{f(1)}10^{f(2)}1\cdots,$$
a computable bijection from elements of $\wtow$ to infinite subsets of $\omega$.
They also use the same notation for the map $\rho:{}^{<\omega}\omega\to{}^{<\omega}2$ where 
    $$\rho(\langle n_0,\cdots,n_k\rangle)=0^{n_0}1\cdots0^{n_k}1.$$
From these they define maps $\rho^*$ and $\rho_*$ which are useful for converting between meagre sets in $\mathcal M$ and in $\mathcal M_C$:
    $$\rho^*(M)=\rho^{-1}[M],$$
    $$\rho_*(M)=\rho[M]\cup\{x\in\ttow:x\text{ is finite}\}.$$
These maps preserve meagreness.
Codes for their outputs can be obtained effectively from codes for their inputs:
In the case of $\rho^*$, map
    $$\langle T_n\rangle_{n\in\omega}\mapsto\langle\rho^{-1}[T_n]\rangle_{n\in\omega}.$$
It is slightly more complicated for $\rho_*$:
for each $n$, let $$\tilde T_n=\{\rho(\sigma)\hatc0^k:\sigma\in T_n\text{ and }k\in\omega\}\cup\{\tau:|\tau^{-1}(1)|\le n\},$$
and map $\langle T_n\rangle_{n\in\omega}$ to $\langle\tilde T_n\rangle_{n\in\omega}$.

We can use these maps to obtain the desired effective morphisms between Weihrauch problems.

\begin{proposition}
    The morphisms $\opn{Supset}(\mathcal M)\leftrightarrow\opn{Supset}(\mathcal M_C)$ are computable.
    Hence, $\opn{Cof}(\mathcal M)\equiv_s\opn{Cof}(\mathcal M_C)$ and $\opn{Add}(\mathcal M)\equiv_s\opn{Add}(\mathcal M_C)$
\end{proposition}

\begin{proof}
    In the forwards direction, let $\Phi_\inst$ map a meagre set $M$ to $\rho_*(M)$ and $\Phi_\sol$ map $M$ to $\rho^*(M)$.
    We obtain the morphism if $\rho_*(M)\subseteq\tilde M$ implies that $M\subseteq\rho^*(\tilde M)$.
    This holds as $\rho$ never outputs finite elements of $\ttow$, so $\rho^*\circ\rho_*$ is the identity as $\rho$ is injective.

    In the reverse direction, let $\Phi_\inst$ map a meagre set $M$ to $\rho^*(M)$ and $\Phi_\sol$ map $M$ to $\rho_*(M)$.
    Similarly, we need that $\rho^*(M)\subseteq\tilde M$ implies that $M\subseteq\rho_*(\tilde M)$.
    This also follows immediately once we disregard finite elements.

    These maps are computable, as described above, so we obtain the desired relations between mass problems.
\end{proof}

Unfortunately, we do not get a computable equivalence for the remaining problems.
One direction does hold:

\begin{proposition}
    \label{capturebairecantor}
    The morphism $\opn{Capture}(\mathcal M)\to\opn{Capture}(\mathcal M_C)$ is computable.
    Hence, $\opn{Cov}(\mathcal M)\le_s\opn{Cov}(\mathcal M_C)$ and $\opn{Non}(\mathcal M_C)\le_s\opn{Non}(\mathcal M)$
\end{proposition}

\begin{proof}
    Let $\Phi_\inst$ map a function $f$ to $\rho(f)$, and let $\Phi_\sol$ map a meagre set $M$ to $\rho^*(M)$.
    Clearly, if $\rho(f)\in M$ then $f\in\rho^*(M)$ by definition.
    These maps are computable.
\end{proof}

There is an effective morphism $\opn{Capture}(\mathcal M_C)\to\opn{Capture}(\mathcal M)$ described by \cite[Lemma 3.7]{GKT:2019}.
They want to use $\rho^{-1}$ for instances and $\rho_*$ for solutions, but this fails to be computable as there is no computable way of
telling whether an instance is infinite, i.e.\ an element of the range of $\rho$.
This is not a proof that the mass problems differ in Muchnik degree, but a result to the contrary would be surprising.
We can get one of the corresponding Medvedev reductions as the solution map here is indeed computable:

\begin{proposition}
    $\opn{Cov}(\mathcal M)\equiv_s\opn{Cov}(\mathcal M_C)$.
    That is, every sequence of computable codes for meagre sets in Baire space which together cover $\wtow\cap\Delta^0_1$
    computes a sequence of computable codes for meagre sets in Cantor space which together cover $\ttow\cap\Delta^0_1$, and vice versa.
\end{proposition}

\begin{proof}
    One direction was shown in Proposition \ref{capturebairecantor}.
    For the other, let ${\langle M_e\rangle_{e\in\omega}\in\opn{Cov}(\mathcal M)}$.
    We argue that $\langle\rho_*(M_e)\rangle_{e\in\omega}\in\opn{Cov}(\mathcal M_C)$.
    Fix $x\in\ttow$ computable. If $x$ is finite, then it is a member of each $\rho_*(M_e)$.
    If $x$ is infinite, then $\rho^{-1}(x)$ is defined and is hence contained in some $M_e$, 
    then by the definition of $\rho_*$, we will have $x\in\rho_*(M_e)$ and we are done.
\end{proof}

Thankfully, the direction we do have for $\opn{Non}(\mathcal M)$ and $\opn{Non}(\mathcal M_C)$ 
is the one we need to maintain the structure of Cicho\'n's diagram.
All that remains is to connect the mass problems for the meagre and null ideals.
One morphism \cite{GKT:2019} makes use of the following well-known fact, for which we give the author's preferred proof:

\begin{proposition}
    \label{partmeagernull}
    There is a meagre set $M$ with a computable code and a null set $N$ with a computable code which form a partition of $\ttow$.
\end{proposition}

\begin{proof}
    Let $P=\{p_i:i\in\omega\}$ be a computable listing of a countable dense subset of $\ttow$, 
    for example, the binary representations of the rational numbers in $[0,1]$.
    For each $i$ and $n$, let $U_{i,n}=[p_i\restriction(i+n)]$.
    We have for each $n$ that $V_n=\bigcup_iU_{i,n}$ is an open cover of $P$ with measure bounded by $\sum_i2^{-(i+n)}=2^{-n}$.
    We also have that the complement of $V_n$ must be nowhere dense, as $V_n$ is dense open.
    Thus, the set $\bigcap_nV_n$ will be null, and its complement is a countable union of nowhere dense sets and is hence meagre.
\end{proof}

We use this in order to map points in Cantor space to null and meagre sets, and this satisfies our desired properties for the following morphism:

\begin{proposition}
    The morphism $\opn{Capture}(\mathcal M_C)\to\opn{Pass}(\mathcal N)$ is computable.
    Hence, $\opn{Cov}(\mathcal M_C)\le_s\opn{Non}(\mathcal M)$ and $\opn{Non}(\mathcal N)\le_s\opn{Cov}(\mathcal M_C)$.
\end{proposition}

\begin{proof}
    By Proposition \ref{partmeagernull}, fix computable meagre and null sets $M$ and $N$ respectively which partition $\ttow$.
    Let $\triangle$ be the symmetric difference operation on members of $\ttow$, and consider the induced operation with a set on one side:
        $$x\triangle Y=\{x\triangle y:y\in Y\}.$$
    Let $\Phi_\inst$ take a point $x$ and map it to $x\triangle N$.
    Similarly, let $\Phi_\sol$ take a point $y$ and map it to $y\triangle M$.

    Firstly, note that these maps preserve measure and category, as for any $x$ and basic open set $[\sigma]$ we have
        $$x\triangle[\sigma]=[(x\!\restriction\!|\sigma|)\triangle\sigma],$$
    a basic open set with the same measure.

    If $x,y\in\ttow$ are computable, and $\Phi_\inst^x=x\triangle N\notni\,y$ then $x\triangle y\notin N$ 
    so $x\triangle y\in M$ and $x\in y\triangle M=\Phi_\sol^y$.
\end{proof}

The reduction in the following is shown in \cite{GKT:2019} indirectly via several reductions through other Weihrauch problems.
The tools they use are not relevant for this section of this thesis,
so we suggest that the reader verifies through their proof that this morphism is computable.

\begin{proposition}
    The morphism $\opn{Supset}(\mathcal M_C)\to\opn{Supset}(\mathcal N)$ is computable.
    Hence, $\opn{Cof}(\mathcal M_C)\le_s\opn{Cof}(\mathcal N)$ and $\opn{Add}(\mathcal N)\le_s\opn{Add}(\mathcal M_C)$.
    \qed
\end{proposition}

We can thus observe that all the relations in Cicho\'n's diagram hold in this mass problem analogy.

\begin{figure}[H]
    \centering
    \begin{tikzpicture}[node distance = 1em and 1em]
        \node (addM) {$\opn{Add}(\mathcal M)$};
        \node (covM) [right = of addM] {$\opn{Cov}(\mathcal M)$};
        \node (b) [above = of addM] {$\mathcal B$};
        \node (d) [above = of covM] {$\mathcal D$};
        \node (nonM) [above = of b] {$\opn{Non}(\mathcal M)$};
        \node (cofM) [above = of d] {$\opn{Cof}(\mathcal M)$};
        \node (addMC) [left = of addM] {$\opn{Add}(\mathcal M_C)$};
        \node (addN) [left = of addMC] {$\opn{Add}(\mathcal N)$};
        \node (nonMC) [left = of nonM] {$\opn{Non}(\mathcal M_C)\,$};
        \node (covN) [left = of nonMC] {$\opn{Cov}(\mathcal N)\,$};
        \node (covMC) [right = of covM] {$\opn{Cov}(\mathcal M_C)$};
        \node (nonN) [right = of covMC] {$\opn{Non}(\mathcal N)$};
        \node (cofMC) [right = of cofM] {$\,\opn{Cof}(\mathcal M_C)$};
        \node (cofN) [right = of cofMC] {$\opn{Cof}(\mathcal N)\;$};

        \draw[double] (addMC) -- (addM); \draw[double] (covMC) -- (covM); \draw[double] (cofMC) -- (cofM); \draw[->] (nonMC) -- (nonM);
        \draw[->] (addN) -- (addMC); \draw[->] (addN) -- (covN); \draw[->] (addM) -- (covM); \draw[->] (addM) -- (b);
        \draw[->] (b) -- (nonM); \draw[->] (covM) -- (d); \draw[->] (covN) -- (nonMC); \draw[->] (d) -- (cofM);
        \draw[->] (b) -- (d); \draw[->] (nonM) -- (cofM); \draw[->] (covMC) -- (nonN);
        \draw[->] (nonN) -- (cofN); \draw[->] (cofMC) -- (cofN);
    \end{tikzpicture}
    \caption{Cicho\'n's diagram for mass problems. Arrows indicate Medvedev reductions. Double lines indicate Medvedev equivalence.}
    \label{cichondiagrammp}
\end{figure}

\subsection{Relationship with the forcing analogy}

Showing there is no reduction between two of these mass problems is difficult.
In the previous section we used the fact that there is a nonlow hyperimmune-free degree to show that $\mathcal T\ngeq_s\mathcal B$.
There may be hope to use a similar non-reduction between well-studied classes of oracles to show other non-reductions between 
mass problem analogues of cardinal characteristics.

The first step towards this is to find lower bounds for the complexity of several of our mass problems.
The forcing analogy looked at by Rupprecht \cite{Rup:2010-1}, Brendle et al.\ \cite{BBNN:2015}, and in \cite{GKT:2019}
is a natural first place to look.

\begin{definition}
    Members of $\high(\opn{Supset}(\mathcal M))$ are referred to as \emph{meagre engulfing}.
    That is, meagre engulfing oracles are those that compute a code for a meagre set which contains all computable meagre sets.
\end{definition}

\begin{definition}
    Members of $\high(\opn{Capture}(\mathcal M))$ are referred to as \emph{weakly meagre engulfing}.
    That is, weakly meagre engulfing oracles are that which compute a code for a meagre set which contains all computable reals.
\end{definition}

\begin{remark}
    \label{wk1genmeagre}
    Members of $\high(\opn{Pass}(\mathcal M))$ are exactly the oracles that compute a weakly 1-generic set.
    Members of $\high(\opn{Spill}(\mathcal M))$ are exactly the oracles that are not low for weak 1-genericity.
    \qed
\end{remark}

For this case with meagre sets, we do indeed get lower bounds from these definitions.
We show that for each Weihrauch problem $A$ that in \cite{GKT:2019} are associated with the meagre ideal we get $\high(A)\le_s\massp(A)$:

\begin{proposition}
    \label{meagreforcingmp}
    The following Medvedev relations hold:
    \begin{enumerate}
        \item[a)] $\high(\opn{Supset}(\mathcal M))\le_s\opn{Cof}(\mathcal M)$. 
            That is, every cofinal sequence in ${\langle\mathcal M\cap\Delta^0_1,\subseteq\rangle}$ is meagre engulfing.
        \item[b)] $\high(\opn{Capture}(\mathcal M))\le_s\opn{Cov}(\mathcal M)$. 
            That is, every unbounded sequence in ${\langle\mathcal M\cap\Delta^0_1,\subseteq\rangle}$ is weakly meagre engulfing.
        \item[c)] $\high(\opn{Spill}(\mathcal M))\le_s\opn{Add}(\mathcal M)$. 
            That is, no sequence of computable codes for meagre sets with the property that every computable real 
            is an element of some member is low for weak 1-genericity.
        \item[d)] $\high(\opn{Pass}(\mathcal M))\le_s\opn{Non}(\mathcal M)$. 
            That is, every sequence of computable reals with the property that no computable meagre set 
            contains the whole sequence computes a weakly 1-generic set.
    \end{enumerate}
\end{proposition}

\begin{proof}
    Given a sequence of codes for meagre sets $\langle M_e\rangle_{e\in\omega}=\langle\langle T_{e,k}\rangle_{k\in\omega}\rangle_{e\in\omega}$,
    note that the union $\bigcup_{e\in\omega}M_e$ is coded by $\langle T_{e,k}\rangle_{e,k\in\omega}$.
    This reduction suffices for (a), (b), and (c):

    For (a), if $\langle M_e\rangle\in\opn{Cof}(\mathcal M)$ then each computable meagre set is contained in some $M_e$ and thus in their union.
    For (b), if $\langle M_e\rangle\in\opn{Cov}(\mathcal M)$ then each computable real is contained in some $M_e$ and thus in their union.
    For (c), if $\langle M_e\rangle\in\opn{Add}(\mathcal M)$ then no computable meagre set contains every $M_e$, 
    hence no computable meagre set contains their union.

    For (d), recall that the weakly 1-generic degrees are precisely the hyperimmune degrees.
    We know from Proposition \ref{hypimm<bd} and Proposition \ref{CapMtoDom} that $\hypimm\le_s\mathcal B\le_s\opn{Non}(\mathcal M)$.
\end{proof}

\begin{definition}
    Members of $\high(\opn{Supset}(\mathcal N))$ are referred to as \emph{null engulfing}.
    That is, null engulfing oracles are those that compute a code for a null set that contains all computable null sets.
    The term \emph{Schnorr engulfing} has also been used.
\end{definition}

\begin{definition}
    Members of $\high(\opn{Capture}(\mathcal N))$ are referred to as \emph{weakly null engulfing}.
    That is, weakly null engulfing oracles are those that compute a code for a null set that contains all computable reals.
    The term \emph{weakly Schnorr engulfing} has also been used.
\end{definition}

\begin{remark}
    \label{schnorrrandomnull}
    Members of $\high(\opn{Pass}(\mathcal N))$ are exactly the oracles that compute a Schnorr random set.
    Members of $\high(\opn{Spill}(\mathcal N))$ are exactly the oracles that are not low for Schnorr randomness.
    \qed
\end{remark}

Obtaining an analogous result to Proposition \ref{meagreforcingmp} would not be simple.
Unlike the case for meagre sets, given a sequence of codes for null sets, it is not straightforward to obtain a code for their union, 
or even for a null set containing their union.
We also cannot rely on a collapse in relations between highness classes, as in the case with hyperimmune and weakly 1-generic degrees
to show a sequence of functions not contained in any computable null set would compute a Schnorr random.

    \newpage
    \section{A $\Pi^0_1$-class which is MED}
\label{schrittesser}

Aside from the question of their cardinality, MAD families have been of interest to set theory for a long time.
Their existence, assuming the axiom of choice, is an easy application of Zorn's lemma, 
so natural questions arise concerning the situation without choice and about the definability of MAD families.

Mathias \cite{Mat:1977} in 1977 showed that the existence of a Mahlo cardinal suffices to construct a model of 
ZF+DC+``There do not exist any MAD families''.
He asked how far the consistency strength could be lowered, and no progress on the problem was made for a long time.
T\"ornquist \cite{Toe:2018} in 2015 showed that Solovay's model has the desired properties and hence an inaccessible cardinal suffices.
This question was answered definitively by Horowitz and Shelah \cite{HS:2019}, who show that no large cardinal assumption is required;
that ZF+DC+``There do not exist any MAD families'' is equiconsistent with ZFC.

Mathias \cite{Mat:1977} also showed that there are no analytic MAD families.
In some sense, this is the best result possible, Miller and Kunen \cite{Mil:1989} showed that assuming $V=L$, there is a coanalytic MAD family.
T\"ornquist \cite{Toe:2009} later gave a simple proof that the existence of a $\mathbf\Sigma^1_2$ MAD family 
implies the existence of a coanalytic ($\mathbf\Pi^1_1$) MAD family, more easily proving the result of Miller and Kunen.

MAD families in other spaces have been considered, notably MAD families in $\wtow$.
Much of this thesis will concern such families, 
but previously a stronger notion has been more studied and is more closely related to MAD families of sets of natural numbers.
A \emph{van Douwen MAD family} is a subset of $\wtow$ where no two elements agree infinitely often, 
and every partial function $h:\;\subseteq\!\!\omega\to\omega$ agrees infinitely often with some member.
Replacing partial functions with total functions in the definition yields the notion of \emph{maximal eventually different (MED) families}.
The question of the existence of van Douwen MAD families was a longstanding problem.
Zhang \cite{Zha:1999-1} showed that they exist assuming Martin's Axiom.
Raghavan \cite{Rag:2010} showed that their existence is provable in ZFC.

Kastermans et al.\ \cite{KSZ:2008} investigate a different, stronger version of maximal eventually different families,
showing they cannot be analytic.
They ask whether the usual, weaker condition can be analytic.
This was answered positively by Horowitz and Shelah \cite{HS:2024} who construct a Borel MED family.
Schrittesser \cite{Sch:2017} improved this to the best possible result, proving the following:

\begin{theorem}
    \label{pi01med}
    There is a MED family that is effectively closed.
\end{theorem}

We present Schrittesser's \cite{Sch:2018} updated construction with a view to making the argument more accessible and filling in some details.
We then build on this result and show that the computable members of this $\Pi^0_1$ MED family are in fact MED for the computable functions.

\vspace*{1em}

\noindent It is easy to create a large family of eventually different functions by taking every function and incorporating redundancy:
Identifying finite strings with natural numbers, for each function $g\in\wtow$, we consider the map $n\mapsto g\restriction n$.
The family of these maps is eventually different, but it is certainly not maximal, and it is difficult to make changes to create a maximal family.
By also adding a second parameter, we will make multiple attempts for each $g$ to be infinitely often equal to a member of our family.
We encode $g\in\wtow$ and $c\in\ttow$ together as a single function $\alt gc$.

Eventually, in order to make these functions MED, we will make changes to each $\alt gc$ 
so that the altered function $\smash{\alth gc}$ is infinitely equal often to $g$ if some conditions are met.
We will either construct $h,d$ such that $\smash{\alth hd}=^\infty g$ 
or show that there must be some binary sequence $c$ meeting some conditions which imply $\alth gc=^\infty g$.

The key property is that the inputs where $c$ takes the value $1$ will impose a constraint on which greater inputs can take the value $1$.
We say that $c$ is \emph{self-refining}.
In particular, for all such $c$, the family of their supports will be almost disjoint.
From $g$ and the support of a self-refining $c$, we build a set $B(g,c)$ on which we will make the changes to $\alth gc$.

\subsection{Schrittesser's construction}

Much of the notation in this section has been changed from the original text, so as to be easier to read.
Throughout this proof, the sets $\omega$, ${}^{<\omega}2$, ${}^{<\omega}\omega$, 
and $\cup_{\ell\in\omega}({}^\ell2\times{}^\ell\omega)$ will be identified via fixed computable bijections.
To remind the reader of this, $\approx$ will often be used in place of $=$ where two types have been identified, 
and $\#$ may be used to directly denote the coding function.

When discussing the binary representation of a number, we use the reverse convention for place value, for example: 
$001_b$ is the binary representation for four, and $0101_b$ is the binary representation for ten.
Let $\bin(n)$ be the binary representation of $n$ as a string following this convention.

We use the following definitions when working with binary strings and sequences in the construction of the MED family:

\begin{definition}
    \label{selfref}
    For each $\sigma\in{}^{<\omega}2$, define
        $$I_{\sigma}=\{n:\sigma\prec\bin(n)\}.$$

    We say that $\sigma$ is \emph{self-refining} if when 
    we have elements $r<s$ such that $\sigma(r)=\sigma(s)=1$ then $s\in I_{\sigma\restriction r+1}$.

    A sequence $c\in\ttow$ is \emph{self-refining} if each initial segment $c\restriction n$ is self-refining.
\end{definition}

As an example, $I_{10}=\{10_b,101_b,1001_b,1011_b,\cdots\}=\{1,5,9,13,\cdots\}$, 
and ${\sigma=100001}$ is self-refining as $\sigma^{-1}[1]=\{0,5\}$ and $5=101_b\in I_{10}$.

A key property of self-refining sequences is that the collection of their supports is AD.
This will be used to create another family of AD sets, which will be essential in ensuring maximality of the final ED family
without making any two functions infinitely often equal.

\begin{claim}
    The family $\{c^{-1}[1]:c\text{ self-refining}\,\}$ is AD.
\end{claim}

\begin{proof}
    Observe that if $\sigma\preceq\tau$ then $I_\tau\subseteq I_\sigma$, 
    and that if $\sigma$ and $\tau$ have the same length but are not equal then $I_\sigma\cap I_\tau=\emptyset$.

    If $c\neq d$ then there is $n$ such that $c\restriction n\neq d\restriction n$ and so  
    $c^{-1}[1]\subseteq^*I_{c\restriction n}$ and $d^{-1}[1]\subseteq^*I_{d\restriction n}$ which are disjoint.
\end{proof}

Now we can work on constructing the MED family.

\begin{definition}
    \label{altgc}
    For every pair of functions $g\in\wtow$ and $c\in\ttow$, define $\alt gc:\omega\to\omega$ by
        $$\alt gc(n)\approx(g\restriction n,c\restriction n)$$
    where we identify $\bigcup_\ell({}^\ell\omega\times{}^\ell2)$ with $\omega$.
    Similarly, if $\sigma\in{}^{<\omega}2$ then let $\alt g\sigma$ be the partial function defined as above on $n<|\sigma|$.
\end{definition}

While this family is ED, it is far from maximal.
For each $g$ that is almost everywhere different from each $\alt hd$, we will change some $\alt gc$ to be equal in value to $g$ on an infinite set.

First, we want some condition that implies $g$ is infinitely often equal to some $\alt hd$.
Consider the following relation $\prec^g$,
a restriction of the usual ordering on $\omega$ to where the strings identified with the images under $g$ are compatible.

\begin{definition}
    \label{precg}
    Let $g\in\wtow$, and $n_0,n_1\in\omega$.
    Define $(\eta_i,\tau_i)\approx g(n_i)$ for each $i\in\{0,1\}$.
    Define $n_0\prec^gn_1$ if:
    \begin{itemize}
        \item $n_0<n_1$,
        \item $|\eta_i|=|\tau_i|=n_i$ for each $i\in\{0,1\}$,
        \item $\eta_0\prec\eta_1$ and $\tau_0\prec\tau_1$.
    \end{itemize}
\end{definition}

If an infinite set $I\subseteq\omega$ is linearly ordered by $\prec^g$ 
then the $\eta,\tau$ associated with $g(n)$ for $n\in I$ are compatible and their union gives a pair $h,d$.
Such $h,d$ would satisfy $\alt hd\restriction I=g\restriction I$ as for each $n\in I$ there are $\eta,\tau$ such that 
    $$\alt hd(n)\approx(h\restriction n,d\restriction n)=(\eta,\tau)\approx g(n).$$

If there is no such $I$, then we want to alter some $\alt gc$ to be infinitely often equal to $g$.
To avoid conflicts we would only want $\alt gc$ to be equal on a thin set, $B(g,c)$:

\begin{definition}
    \label{Bgc}
    For every pair of functions $g\in\wtow$ and $c\in\ttow$, define
        $$B(g,c)=\{2\cdot\#(g\restriction n):c(n)=1\}\smallsetminus\{n:g(n)=\alt gc(n)\}.$$
\end{definition}

We use that the supports of self-refining sequences are AD to show that the sets $B(g,c)$ inherit this property;
they can thus be used to avoid introducing infinitely often equal functions in our ED family.

\begin{claim}
    The family $\{B(g,c):c\text{ self-refining}\,\}$ is AD.
\end{claim}

\begin{proof}
    Recall that ${}^{<\omega}\omega$ and $\omega$ are identified via a computable bijection $\#$.
    Note that for any $g$, the map $n\mapsto 2\cdot\#(g\restriction n)$ is injective.
    As $B(g,c)$ is a subset of the image of $c^{-1}[1]$ under that map, if $c$ and $d$ differ then $B(g,c)$ and $B(h,d)$ will be AD.
    If $g$ and $h$ differ, then the sets $\{2\cdot\#(g\restriction n):n\in\omega\}$ and $\{2\cdot\#(h\restriction n):n\in\omega\}$ are AD, and so
    $B(g,c)$ and $B(h,d)$ will be AD.
\end{proof}

To simplify the notation and break up the combinatorics in the maximality proof below,
it is helpful to define $I^g_\sigma$ to be the image of $I_\sigma$ under the above map:

\begin{definition}
    \label{Igs}
    For $g\in\wtow$ and $\sigma\in{}^{<\omega}2$, define:
        $$I^g_\sigma=\{2\cdot\#(g\restriction n):n\in I_\sigma\}.$$
\end{definition}

\begin{remark}
    Fix $\ell\in\omega$.
    For any $g,h\in\wtow$ and $\sigma,\tau\in{}^\ell2$, the intersection $I^g_\sigma\cap I^h_\tau$ is finite.
\end{remark}

\begin{proof}
    Let $\langle g,\sigma\rangle$ and $\langle h,\tau\rangle$ be distinct.
    If $g\neq h$ then for all but finitely many $n$ the initial segments $g\restriction n$ and $h\restriction n$ will be distinct.
    Thus, as $\eta\mapsto2\cdot\#(\eta)$ is injective, $I^g_\sigma\cap I^h_\tau$ is finite.
    If $g=h$, then $\sigma\neq\tau$ so $I_\sigma\cap I_\tau=\emptyset$.
    Thus, as $n\mapsto2\cdot\#(g\restriction n)$ is injective, $I^g_\sigma\cap I^h_\tau$ is finite.
\end{proof}

Now everything is in place to define our MED family:

\begin{definition}
    \label{althgc}
    For every pair of functions $g\in\wtow$ and $c\in\ttow$, define $\alth gc:\omega\to\omega$ by
        $$\alth gc(n)=\begin{cases}
            g(n)&n\in B(g,c)\text{, }B(g,c)\restriction n\text{ is a}\prec^g\!\!\text{-antichain, and }c\restriction n\text{ is self-refining,}\\
            \alt gc(n)&\text{otherwise.}
        \end{cases}$$
    Let $\mathcal E=\{\alth gc:g\in\wtow,c\in\ttow\}$. 
\end{definition}

\setlength{\tabcolsep}{0.8em} {\renewcommand{\arraystretch}{1.4}
\begin{figure}[h]
    \centering
    \begin{tabular}{|l|l|}
        \hline
        $I_\sigma$          & Def. \ref{selfref} \\ \hline
        self-refining       & Def. \ref{selfref} \\ \hline
        $\alt gc$           & Def. \ref{altgc} \\ \hline
        $\prec^g$           & Def. \ref{precg} \\ \hline
        $B(g,c)$            & Def. \ref{Bgc} \\ \hline
        $I^g_\sigma$        & Def. \ref{Igs} \\ \hline
        $\alth gc$          & Def. \ref{althgc} \\ \hline
    \end{tabular}
    \caption{References for notation and terminology used in this proof.}
\end{figure}

For each $g\in\wtow$ we want either some pair $h,d$ such that $\alt hd=^\infty g$ and we preserve this in $\alth hd$,
or some $c$ such that $\alth gc=^\infty g$ due to the changes made on inputs in $B(g,c)$.
The following lemma shows us that this construction ensures that one of these will occur, and thus maximality of $\mathcal E$ will follow.

\begin{lemma}
    \label{2.7}
    For any $g\in\wtow$, one of the following holds:
    \begin{enumerate}
        \item There is an infinite set $I$ which is linearly ordered by $\prec^g$.
            The definition of $\prec^g$ thus gives functions $h\in\wtow$ and $d\in\ttow$ with $g\restriction I=\alt hd\restriction I$.
            We get that $I\cap B(h,d)$ is finite.
        \item There is a self-refining function $c$ which takes the value $1$ infinitely often 
            and it is such that $B(g,c)$ is a $\prec^g$-antichain. 
    \end{enumerate}

\end{lemma}

\begin{proof}
    We prove the first case under the assumption that a statement holds, and if not then we obtain the second case.

    Define $C=\{\sigma\in{}^{<\omega}2:|\sigma|=0\text{ or }\sigma\text{ ends with }1\}$.
    In this proof, where $\sigma$ and $n_0$ are specified, let $\sigma_1=\sigma\hatc0^{n_0-|\sigma|}\hatc1$.
    We consider the assumption:
    \begin{equation}
        \label{2.7hyp}
        (\exists\sigma\in C)(\forall n_0\in I^g_{\sigma})[g(n_0)\neq \alt g{\sigma_1}(n_0)\land(\exists n_1\in I^g_{\sigma_1})(n_0\prec^gn_1)]
    \end{equation}
    Fix $\sigma$ witnessing the first quantifier in (\ref{2.7hyp}), and let $n_0=\min I^g_{\sigma}$.
    From (\ref{2.7hyp}) we obtain $n_1\in I^g_{\sigma_1}$ such that $n_0\prec^g n_1$.
    Recursively choosing $\sigma_{k+1}=\sigma_k\hatc0^{n_k-|\sigma_k|}\hatc1$ and $n_{k+1}\in I^g_{\sigma_{k+1}}$ such that $n_k\prec^gn_{k+1}$,
    we define a set $I=\{n_k:k\in\omega\}$ which is linearly ordered by $\prec^g$.
    This gives us $h$ and $d$ such that $g\restriction I=\alt hd\restriction I$.

    As $\alt hd$ is equal to $g$ on $I$, but $g$ is distinct from $\alt g{\sigma_1}$ at each point in $I$, we have that 
    $I^h_{d\restriction n_0}\cap I^g_{\sigma_1\restriction n_0}$ is finite.
    Thus, as $B(h,d)\subseteq^* I^h_{d\restriction n_0}$ and $I\subseteq^* I^g_{\sigma_1\restriction n_0}$ we have that $B(h,d)\cap I$ is finite.

    \vspace*{1em}

    \noindent Now assume that (\ref{2.7hyp}) fails, and thus we have the following:
    \begin{equation}
        \label{2.7hypneg}
        (\forall\sigma\in C)(\exists n_0\in I^g_{\sigma})[g(n_0)=\alt g{\sigma_1}(n_0)\lor(\forall n_1\in I^g_{\sigma_1})\lnot(n_0\prec^gn_1)]
    \end{equation}
    Let $\sigma_0=\emptyset$ and $n_0$ witnessing the second quantifier in (\ref{2.7hypneg}).
    Similarly to above, recursively choose $n_k$ and $\sigma_{k+1}$ such that $\sigma_{k+1}=\sigma_k\hatc0^{n_k-|\sigma_k|}\hatc1$, 
    and so they witness (\ref{2.7hypneg}).
    Let $c=\bigcup_{k\in\omega}\sigma_k$.
    By the choices of $\sigma_{k+1}$, observe that $c$ is self-refining and takes the value $1$ infinitely often.
    Whenever $n_k$ satisfies the first disjunct, we have $n_k\notin B(g,c)$. 
    We get that $B(g,c)=\{n_k:k>0\text{ and }g(n_k)\neq\alt g{\sigma_{k+1}}(n_k)\}$, 
    which is a $\prec^g$-antichain by the second disjunct for each member and that for any two indices $\ell\le k$ we have $n_k\in I^g_{\sigma_\ell}$.
\end{proof}

From Lemma \ref{2.7} it is not hard to show that $\mathcal E$ is maximal:

\begin{claim}
    The family $\mathcal E$ is MED.
\end{claim}

\begin{proof}
    First, we show that $\mathcal E$ is ED.
    Let $f_0,f_1\in\mathcal E$ be distinct members.
    Let $g,c$ and $h,d$ such that $f_0=\alth gc$ and $f_1=\alth hd$.
    By definition, were $f_0$ and $f_1$ to agree on infinitely many points, 
    then all but finitely many of them would be in $B(g,c)\cup B(h,d)$.
    By symmetry, it suffices to show that
        $$X=\{n\in B(g,c)\smallsetminus B(h,d):\alth gc(n)=\alt hd(n)\}$$
    is finite, as $\alth hd$ agrees with $\alt hd$ on the complement of $B(h,d)$.
    $X$ is a subset of $B(g,c)$ so $\alth gc\restriction X=\alt gc\restriction X$.
    If $X$ is infinite, then $\alt gc$ and $\alt hd$ agree on an infinite set,
    and so $(g,c)=(h,d)$ and $f_0=f_1$, contradicting our assumption.

    \vspace*{1em}

    \noindent Fix $g\in\wtow$.
    If the first case of Lemma \ref{2.7} holds, 
    then there is an infinite set $I$ and $h,d$ such that $g\restriction I=\alt hd\restriction I$ and $I\cap B(h,d)$ is finite.
    As $I\cap B(h,d)$ is finite, $g$ also agrees with $\alth hd$ on infinitely many points.
    
    If the second case holds, then there is $c$ self-refining which takes the value $1$ infinitely often 
    and it is such that $B(g,c)$ is a $\prec^g$-antichain.
    If $B(g,c)$ is infinite, then $g\restriction B(g,c)=\alth gc\restriction B(g,c)$ and we are done.
    If $B(g,c)$ is finite, then there are infinitely many $n$ such that $g(n)=\alt gc(n)$ and so we certainly have that $g=^\infty\alth gc$.
\end{proof}

Finally, we prove that $\mathcal E$ is a $\Pi^0_1$-class.
It will be helpful to extend a few of the above definitions to work with finite strings.
The following are simply approximations of the infinite versions above:

\begin{definition}
    For each $\zeta\in{}^{<\omega}\omega$ and $\sigma\in{}^{<\omega}2$ such that $|\zeta|=|\sigma|$ define
        $$B(\zeta,\sigma)=\{2\cdot\#(\zeta\restriction n):n<|\sigma|\text{ and }\sigma(n)=1\}
            \smallsetminus\{n:\zeta(n)\approx(\zeta\restriction n,\sigma\restriction n)\}.$$
    Also, for $\zeta\in{}^{<\omega}\omega$ and $n_0,n_1\in\omega$ with $(\eta_i,\tau_i)\approx\zeta(n_i)$ for each $i\in\{0,1\}$,
    define $n_0\prec^\zeta n_1$ if
    \begin{itemize}
        \item $n_0<n_1<|\zeta|$,
        \item $|\eta_i|=|\tau_i|=n_i$ for each $i\in\{0,1\}$,
        \item $\eta_0\prec\eta_1$ and $\tau_0\prec\tau_1$.
    \end{itemize}
\end{definition}

We build a tree of approximations to the elements $\alth gc$ of $\mathcal E$:

\begin{proof}[Proof of Theorem \ref{pi01med}]
    Define the tree $T\subseteq{}^{<\omega}\omega$ as the set of $\xi\in{}^{<\omega}\omega$ such that 
    for all odd $k<|\xi|$ we have $|\zeta|=|\sigma|=k$ where $(\zeta,\sigma)\approx\xi(k)$, and for every $n<k$:

    \noindent If all of the following conditions hold:
    \begin{itemize}
        \item $\sigma$ is self-refining,
        \item $B(\zeta,\sigma)\cap n$ is a $\prec^\xi$-antichain, and
        \item $n\in B(\zeta,\sigma)$
    \end{itemize}
    then $\xi(n)=\zeta(n)$,
    otherwise $\xi(n)\approx(\zeta\restriction n,\sigma\restriction n)$.

    The set of self-refining strings is computable, and each set $B(\zeta,\sigma)$ is finite, so the relevant properties are decidable.
    Hence $T$ is computable.

    As each set $B(g,c)$ contains only even values, 
    each element $\alth gc\in\mathcal E$ has value $\alt gc(n)$ on all odd values $n$ so we can recover finite segments of $g$ and $c$ on those values.
    Fixing arbitrary $\xi\prec\alth gc$, for odd $k<|\xi|$ and $(\zeta,\sigma)\approx\xi(k)$ 
    we indeed have that the above requirement holds by the definition of $\alth gc$.
    Hence each $\alth gc\in[T]$.

    If we take a path $f\in[T]$ we have for all odd $k$ that $(\zeta_k,\sigma_k)\approx f(k)$ with $|\zeta_k|=|\sigma_k|=k$
    and for odd $\ell>k$ that $\zeta_k\prec\zeta_\ell$ and $\sigma_k\prec\sigma_\ell$.
    Defining $g=\bigcup_{k\text{ odd}}\zeta_k$ and $c=\bigcup_{k\text{ odd}}\sigma_k$ we have $f=\alth gc$ by the above requirement,
    and so $f\in\mathcal E$.
    Hence $\mathcal E=[T]$.
\end{proof}

\subsection{Computable members of $\mathcal E$}
\label{medrec}

While the computability theoretic notion of MED families defined in Section \ref{sf} is useful in that it can be coded as a single object,
it has various differences from the set theoretic notion.
In particular, it enforces an arbitrary ordering on the family, which is encoded in the object itself.
Another computability theoretic approach is to look only at the computable members of a family, disregarding the rest.
If we are looking at a kind of family with a strict internal structure, then the descriptive complexity may be interesting.
We show that Schrittesser's construction witnesses the existence of a $\Pi^0_1$ family whose members are MED relative to computable functions:

\begin{theorem}
    \label{effectivesch}
    Let $\mathcal E$ be the $\Pi^0_1$ MED family constructed in the proof of Theorem~\ref{pi01med}.
    The computable members of $\mathcal E$ are MED for computable functions.
    That is, each computable function is infinitely often equal to some computable member of $\mathcal E$.
\end{theorem}

Verifying this comes down to showing that for any computable $g$, the witnesses for maximality obtained in Lemma \ref{2.7} will also be computable.
The argument in Lemma \ref{2.7} makes use of elements of the sets $I^g_\sigma$, so we first check that we can easily work with these sets:

\begin{remark}
    \label{Igscomp}
    For any computable $g\in\wtow$ and $\sigma\in{}^{<\omega}2$, the set $I^g_\sigma$ is computable uniformly in $g$ and $\sigma$.
\end{remark}

\begin{proof}
    Recall that ${}^{<\omega}\omega$ and $\omega$ are identified via a computable bijection $\#$.
    Observe that
        $$n\in I^g_\sigma\iff n\text{ even and }\#^{-1}(n/2)\prec g\text{ and }|\#^{-1}(n/2)|\in I_\sigma.$$
    This can be decided uniformly in $g$ and $\sigma$.
\end{proof}

The key idea in the following effective version of Lemma \ref{2.7}
is that while for each $g$ we have no effective way to tell which case in Lemma \ref{2.7} will hold,
in particular whether (\ref{2.7hyp}) holds or not, 
in either case the construction of a function in $\mathcal E$ which is infinitely often equal to $g$ will be computable.
We do not need to be able to know which function in $\mathcal E$ is infinitely often equal to $g$, just that there is at least one.

\begin{lemma}
    \label{efflem}
    For any computable $g\in\wtow$, one of the following holds:
    \begin{enumerate}
        \item There is a computable infinite set $I$ which is linearly ordered by $\prec^g$.
            The definition of $\prec^g$ thus gives computable functions $h\in\wtow$ and $d\in\ttow$ with $g\restriction I=\alt hd\restriction I$.
            We get that $I\cap B(h,d)$ is finite and so $g\restriction I=^*\alth hd\restriction I$.
        \item There is a computable and self-refining function $c$ such that $B(g,c)$ is an infinite $\prec^g$-antichain.
            Thus, $g\restriction B(g,c)=\alth gc\restriction B(g,c)$.
    \end{enumerate}
\end{lemma}

\begin{proof}
    We show that the proof of Lemma \ref{2.7} yields computable functions in either case.

    Recall that we define $C=\{\sigma\in{}^{<\omega}2:|\sigma|=0\text{ or }\sigma\text{ ends with }1\}$, and where $\sigma$ and $n_0$ are specified,
    we let $\sigma_1=\sigma\hatc0^{n_0-|\sigma|}\hatc1$.
    Again, consider the following assumption:
    \begin{equation}
        \label{2.7hypeff}
        (\exists\sigma\in C)(\forall n_0\in I^g_\sigma)[g(n_0)\neq\alt g{\sigma_1}(n_0)\land(\exists n_1\in I^g_{\sigma_1})(n_0\prec^gn_1)]
    \end{equation}
    Under the assumption of (\ref{2.7hypeff}), fix $\sigma$ witness the first quantifier.
    As $\sigma$ is finite, it does not matter that we cannot obtain it uniformly in $g$; it adds no computational complexity of its own.
    As the set $I^g_\sigma$ is computable, by Remark \ref{Igscomp}, the following process yields an infinite set $I$:
    Let $n_0=\min I^g_\sigma$.
    This defines $\sigma_1$ as above, so by (\ref{2.7hypeff}), take $n_1\in I^g_{\sigma_1}$ such that $n_0\prec^gn_1$.
    Recursively choosing ${\sigma_{k+1}=\sigma_k\hatc0^{n_k-|\sigma_k|}\hatc1}$ and $n_{k+1}\in I^g_{\sigma_{k+1}}$ such that $n_k\prec^gn_{k+1}$,
    
    we define a computable set ${I=\{n_k:k\in\omega\}}$ which is linearly ordered by $\prec^g$.

    Define $\langle\eta_k,\tau_k\rangle\approx g(n_k)$ for each $k$.
    Let $h=\bigcup_k\eta_k$ and $d=\bigcup_k\tau_k$.
    The functions $h$ and $d$ will be computable as $I$ and $g$ are computable.
    The proof that $g\restriction I=^*\alth hd\restriction I$ needs no further verification from the proof of Lemma \ref{2.7}.

    If (\ref{2.7hypeff}) fails then we have:
    \begin{equation}
        \label{2.7hypnegeff}
        (\forall\sigma\in C)(\exists n_0\in I^g_\sigma)[g(n_0)=\alt g{\sigma_1}(n_0)\lor(\forall n_1\in I^g_{\sigma_1})\lnot(n_0\prec^gn_1)]
    \end{equation}
    Define the sequence $\sigma_k$ recursively as in the proof of Lemma \ref{2.7}.
    The existential statement and computability of $I^g_\sigma$ uniformly in $\sigma$ make sure that each successive $n_k$ can be found.
    By the assumption of (\ref{2.7hypnegeff}) and the computability of $I^g_\sigma$ this process is computable, 
    and so $c=\bigcup_k\sigma_k$ is computable.
    The rest of the verification is unchanged from the proof of Lemma \ref{2.7}.
\end{proof}

Similarly to before, this is sufficient to prove that the computable members of $\mathcal E$ are MED relative to the computable functions:

\begin{proof}[Proof of Theorem \ref{effectivesch}]
    As $\mathcal E$ as a whole is ED, its computable members certainly are.
    Let $g$ be a computable function.
    By Lemma \ref{efflem} we know that for any computable $g\in\wtow$, 
    there are computable functions $h\in\wtow$ and $d\in\ttow$ such that $g=^\infty\alth hd$.
    The definition of $\alth hd$ checks only finite conditions, and so $\alth hd$ is computable.
    Thus, there is a computable member of $\mathcal E$ which is infinitely often equal to $g$.
\end{proof}

    \newpage
    \section{Discussion}

The framework of mass problems that Lempp, Miller, Nies, and Soskova \cite{LMNS:2023} 
used for computability-theoretic analogues of cardinal characteristics 
is very flexible in representing the kinds of families relevant to cardinal characteristics research for various purposes in computability.
We have built on their work to give analogues for several other cardinal characteristics and shown relations to those they defined.

The computability-theoretic MED families and towers of functions discussed in Section \ref{sf} show that there is a rich structure
of these mass problem analogues around the MAD families and towers of sets discussed by in \cite{LMNS:2023}.
In particular, these mass problems do not seem to coincide in the Muchnik or Medvedev degrees with any well-studied notions.
The analogues of the cardinals in Cicho\'n's diagram discussed in Section \ref{mpcichon} witness the flexibility of this framework.
It fits well into the general framework by Greenberg, Kuyper, and Turetsky \cite{GKT:2019}; 
associating mass problems to their Weihrauch problems and the relations follow from computable morphisms.

A key point in both Sections \ref{sf} and \ref{mpcichon} is the contrast between the mass problems analogy of families of sets or functions,
and the highness classes analogy of the corresponding forcing methods in set theory.
Several similarities and differences between the two frameworks can be observed.
In both frameworks, we can observe that many relations collapse once the degrees are necessarily high.
This is often the case as high sets are sufficient to compute all but a finite amount of information from computable sets 
from their indices in a uniform way.
Aside from the absolute complexity of the analogues,
the framework of mass problems allows us to find analogues of cardinal characteristics that highness classes do not.
Cardinal characteristics corresponding to families with internal structure cannot be described 
with the same kind of relational structure as the ones in Cicho\'n's diagram.

In the setting of highness classes, it has proven interesting to consider relations other than Turing reductions.
Kihara \cite{Kih:2017} investigated the case where an analogue of Cicho\'n's diagram is obtained for the 
hyperarithmetical sets and hyperarithmetical reducibility.
Fewer relations collapse in this setting compared to with computable sets, but some of the classes still coincide.
Recent work by Greenberg and Osso \cite{GO:2026} considered the case of Turing reduction modulo the ideal of hyperarithmetical sets.
They show that in this case, even further separations can be obtained, in particular between the analogous versions of high, (strong) meagre engulfing,
and (strong) null engulfing oracles.
This answers several questions asked in \cite[Section 8]{GKT:2019}.
One can ask similar questions in the setting of mass problems.
It would be an interesting direction to see which relations or separations between mass problems could be obtained 
when considering alternatives to the computable sets or functions and the Turing reduction.

The mass problem framework this thesis focused on is flexible but has significant differences from the set theory setting.
The sequences we associate with kinds of set-theoretic families encode the family alongside an arbitrary ordering.
This ordering carries with it its own complexity and may be useful in computing members of other classes, 
yielding results that may have no set-theoretic analogue for a reason not inherent to computability.
Perhaps one can find a way of representing analogues for the set-theoretic family that does not enforce an arbitrary ordering.
One possibility is to consider the paths through a tree; Question \ref{5.3} below is in this direction.

\subsection{Open questions}

Separating mass problems in the Muchnik degrees has proven difficult,
and so there are several questions left open from Sections \ref{sf} and \ref{mpcichon}:

\begin{question}
    \label{qtowers}
    Which, if any, of the following Muchnik relations hold?
    \begin{enumerate}
        \item[a)] $\ts\le_w\tf$: Do maximal towers of functions compute maximal towers of sets?
        \item[b)] $\ts\le_w\ts^-$: Do maximal weak towers of sets compute maximal towers of sets?
        \item[c)] $\tf\le_w\ts^-$: Do maximal weak towers of sets compute maximal towers of functions?
        \item[d)] $\ts^-\le_w\tf$: Do maximal towers of functions compute maximal weak towers of sets?
    \end{enumerate}
\end{question}

A negative answer to any part of Question \ref{qtowers} would, by Proposition \ref{iggeneral} and the known relations between these mass problems,
show that index guessability is not equivalent to computing no maximal tower of sets, and thus would answer a question from \cite{LMNS:2023}.

We can ask a similar question about the analogues of cardinals in Cicho\'n's diagram:

\begin{question}
    Which Muchnik relations in Cicho\'n's diagram for mass problems, Figure \ref{cichondiagrammp}, can be reversed?
\end{question}

Clearly $\bd<_s\mathcal D$ as there are low $\le^*$-unbounded families.
Aside from this, little is known for sure.
It seems very possible that the mass problems in Figure \ref{cichondiagrammp} above $\mathcal D$ are Medvedev equivalent to it.
Highness is enough to compute members of many mass problem analogues of cardinal characteristics, as shown in \cite{LMNS:2023}
with ultrafilter bases and maximal independent families, and the author \cite{McD:2024} with maximal ideal independent families.

In Section \ref{medrec}, we saw that there is a $\Pi^0_1$-class whose computable members are MED for computable functions.
Although there is no $\Pi^0_1$ MAD family (indeed, no analytic MAD family) in the set theory setting,
we can ask about the computable members:

\begin{question}
    \label{5.3}
    Is there a $\Pi^0_1$-class $P$ such that the computable members of $P$ are pairwise AD and for any computable set $R$ there is
    a computable set $X\in P$ such that $X\cap R$ is infinite?
\end{question}

As there is no constraint on the noncomputable members of $P$, 
there is no immediate reason that this would contradict the known results about MAD families in set theory.

    \newpage
    \appendix
\section{Notation and Conventions}
\label{appnot}

\setlength{\tabcolsep}{0.8em} {\renewcommand{\arraystretch}{1.4}
\begin{figure}[h]
    \centering
    \begin{tabular}{|l|l|}
        \hline
        Cardinal characteristics        & $\mathfrak{a,b,c,\cdots}$ \\ \hline
        Cardinals                       & $\kappa,\lambda$ \\ \hline
        Finite binary strings           & $\sigma,\tau,\rho$ \\ \hline
        Finite natural number strings   & $\eta,\zeta,\xi$ \\ \hline
        Infinite binary sequences       & $c,d$ \\ \hline
        Total natural number functions  & $f,g,h,p$ \\ \hline
        Partial natural number functions& $\varphi,\psi,\vartheta$ \\ \hline
        Turing functionals              & $\Phi,\Psi,\Gamma$ \\ \hline
        Ideals of families of reals     & $\mathcal{I,M,N}$ \\ \hline
        Mass problems                   & $\mathcal{A,B,C,\cdots}$ \\ \hline
    \end{tabular}
    \caption{Naming conventions.}
\end{figure}

The collection of infinite sets of natural numbers is written $\infsets$.
The collection of natural number functions is written $\wtow$, and similar spaces notated analogously.

For a relation $R\subseteq X\times X$, we define relations on ${}^\omega\!X$ by $xR^*y$ if $x(n)Ry(n)$ for all but finitely many $n\in\omega$,
and $xR^\infty y$ if $x(n)Ry(n)$ for infinitely many $n\in\omega$.
The latter is used to avoid ambiguity of the form $\not\mathrel{R}^*$.
For example, if $f:n\mapsto 100n$ and $g\mapsto n^2$ then $f\le^*g$.

We refer to a fixed computable bijection $\langle\cdot,\cdot\rangle:\omega\times\omega\to\omega$.

    \newpage

    \addcontentsline{toc}{section}{References}
    \printbibliography
\end{document}